\documentclass[reqno,a4paper,twoside]{amsart}

\usepackage{enumitem}

\setenumerate{label=\textnormal{(\arabic*)}}

\usepackage{amsmath,amssymb,dsfont,verbatim,bm,mathtools,geometry,fge}
\usepackage[mathcal]{euscript}
\usepackage[raggedright]{titlesec}
\usepackage{flafter}
\usepackage{microtype}
\usepackage{quoting}
\usepackage[latin1]{inputenc}
\usepackage{setspace}
\usepackage{mathtools}
\usepackage{tikz}
\usepackage{float,subfig}
\usepackage{amsrefs}
\usepackage[linesnumbered,ruled,vlined]{algorithm2e}
\SetAlgoSkip{smallskip}
\SetAlgoInsideSkip{smallskip}

\usepackage{booktabs}
\usepackage[section]{placeins}
\usepackage{flafter}
\usepackage{longtable, tabu}
\usepackage[foot]{amsaddr}

\newcommand{\ra}[1]{\renewcommand{\arraystretch}{#1}}

\titleformat{\chapter}[display]
{\normalfont\huge\bfseries}{\chaptertitlename\\thechapter}{20pt}{\Huge}
\titleformat{\section}
{\normalfont\Large\bfseries\center}{\thesection}{1em}{}
\titleformat{\subsection}
{\normalfont\large\bfseries}{\thesubsection}{1em}{}
\titleformat{\subsubsection}[runin]
{\normalfont\normalsize\bfseries}{\thesubsubsection}{1em}{}
\titleformat{\paragraph}[runin]
{\normalfont\normalsize\bfseries}{\theparagraph}{1em}{}
\titleformat{\subparagraph}[runin]
{\normalfont\normalsize\bfseries}{\thesubparagraph}{1em}{}
\titlespacing*{\chapter} {0pt}{50pt}{40pt}
\titlespacing*{\section} {0pt}{3.5ex plus 1ex minus .2ex}{2.3ex plus .2ex}
\titlespacing*{\subsection} {0pt}{3.25ex plus 1ex minus .2ex}{1.5ex plus .2ex}
\titlespacing*{\subsubsection}{0pt}{3.25ex plus 1ex minus .2ex}{1.5ex plus .2ex}
\titlespacing*{\paragraph} {0pt}{3.25ex plus 1ex minus .2ex}{1em}
\titlespacing*{\subparagraph} {\parindent}{3.25ex plus 1ex minus .2ex}{1em}


\usepackage{titletoc}
\usepackage{mathrsfs}
\contentsmargin{2.5em}
\dottedcontents{section}[1.5em]{}{2.8em}{1pc}
\dottedcontents{subsection}[3.7em]{}{3.4em}{1pc}
\dottedcontents{subsubsection}[7.6em]{}{3.2em}{1pc}
\dottedcontents{paragraph}[10.3em]{}{3.2em}{1pc}

\newtheorem{theorem}{Theorem}[section]

\newtheorem{proposition}[theorem]{Proposition}
\newtheorem{corollary}[theorem]{Corollary}

\theoremstyle{definition}

\theoremstyle{remark}
\newtheorem{remark}[theorem]{Remark}

\DeclareMathOperator{\Aut}{Aut}

\DeclareMathOperator{\en}{en}

\DeclareMathOperator{\st}{st}

\DeclareMathOperator{\Succ}{Succ}
\DeclareMathOperator{\Pred}{Pred}

\DeclareMathOperator{\dir}{dir}

\DeclareMathOperator{\ord}{ord}

\newcommand{\ov}{\overline}

\DeclareMathOperator{\enF}{enF}

\begin{document}

\title{Increasing the degree of a possible counterexample to the Jacobian Conjecture from 100 to 108}

\author[Jorge A. Guccione]{Jorge A. Guccione$^{1,}$$^2$}
\address{$^1$ Universidad de Buenos Aires. Facultad de Ciencias Exactas y Naturales. Departamento de Matem\'atica. Buenos Aires. Argentina}
\address{$^2$ CONICET-Universidad de Buenos Aires. Instituto de Investigaciones Matem\'aticas ``Luis A. santal\'o'' (IMAS). Buenos Aires. Argentina}
\email{vander@dm.uba.ar}

\author[Juan J. Guccione]{Juan J. Guccione$^{1,}$$^3$}
\address{$^3$ CONICET. Instituto Argentino de Matem\'atica (IAM). Buenos Aires. Argentina}
\email{jjgucci@dm.uba.ar}

\thanks{Jorge A. Guccione and Juan J. Guccione were supported by UBACyT 20020150100153BA (UBA) and PIP 11220110100800CO (CONICET)}

\author[Rodrigo Horruitiner]{Rodrigo Horruitiner$^{4}$}
\address{$^4$Cornell University, Department of Mathematics, 310 Malott Hall, Cornell University, Ithaca, New York, USA.}
\email{rmh322@cornell.edu}

\author[Christian Valqui]{Christian Valqui$^{5,}$$^6$}
\address{$^5$Pontificia Universidad Cat\'olica del Per\'u, Secci\'on Matem\'aticas, PUCP, Av. Universitaria 1801, San Miguel, Lima 32, Per\'u.}

\address{$^6$Instituto de Matem\'atica y Ciencias Afines (IMCA) Calle Los Bi\'ologos 245. Urb San C\'esar.
La Molina, Lima 12, Per\'u.}
\email{cvalqui@pucp.edu.pe}
\thanks{This work was supported by CONCYTEC-FONDECYT within the framework of the contest ``Proyectos de Investigaci\'on B\'asica 2020-01'' [contract number 120-2020-FONDECYT]}

\begin{abstract}
We list all the pairs  $(\deg(P),\deg(Q))$ with $\max\{\deg(P),\deg(Q)\}< 125$ for any hypothetical counterexample to the plane Jacobian Conjecture and discard them all, except the pair $(72,108)$ (and the symmetric pair $(108,72)$), thus we confirm the lower bound of 100 obtained by Moh and raise it up to 108.
\end{abstract}

\maketitle

\setcounter{tocdepth}{2}
\tableofcontents

\section{Introduction}
The plane Jacobian Conjecture (see~\cite{K}) can be rephrased in the following manner:

If a pair of polynomials $(P,Q)$ in $K[x,y]$ satisfies $[P,Q]:= P_x Q_y-P_y Q_x\in K^{\times}$, then there exists a
polynomial inverse to the morphism $(P,Q):K^2\to K^2$.

In~\cite{M} the author discards all possible counterexamples with $\max\{\deg(P),\deg(Q)\}\le 100$ using approximate roots of
polynomials in two variables, but provides a detailed proof only for the smallest case.
In~\cite{GGV1} the shape of the support of possible counterexamples to the Jacobian Conjecture in dimension 2 is described in great detail.
Based on its results, and some results of~\cite{GGV2},
in~\cite{GGV5} a list of possible corners of small counterexample is given. We will adapt the
techniques of~\cite{GGV1}, \cite{GGV2} and~\cite{GGV3} in order to
discard nearly all the cases with $\max\{\deg(P),\deg(Q)\}<125$, which haven't been discarded previously. The only exception is
the case $(\deg P,\deg Q)=(72,108)$, and so, if one manages to discard this case, it would increase the lower bound from $108$ up to $125$.

The article is organized as follows. In section~\ref{seccion 1} we list all the cases with $\max\{\deg(P),\deg(Q)\}<125$, following~\cite{GGV5}.
There are 10 cases, from which we consider 5 to be discarded already by previous work of different authors. If we assume
$\deg(P)<\deg(Q)$, then the remaining 5 cases satisfy
$$
(\deg(P),\deg(Q))\in\{(56,84),(66,99),(72,108),(80,120)\},
$$
where there are two cases with $(\deg(P),\deg(Q))=(72,108)$. In section~\ref{seccion 2} we first discard the case
$\max\{\deg(P),\deg(Q)\}=120$, as a straightforward application of one of our previous results in~\cite{GGV2}.
 Then we discard $\max\{\deg(P),\deg(Q)\}=84$
using an inequality for the intersection number proved in~\cite{GGV6}.
In section~\ref{seccion 3} we apply some transformation of $K[x,y]$ introduced in~\cite{GGV1} in order to reduce the size
of the support of $P$ and $Q$ in all cases except in the discarded case $(\deg(P),\deg(Q))=(80,120)$.
In section~\ref{seccion 4} we use the systems of polynomial equations associated to a possible counterexample as in~\cite{GGV3} in order
to discard the case $(\deg(P),\deg(Q))=(66,99)$ and one of the cases with $(\deg(P),\deg(Q))=(72,108)$. For the other case
with $(\deg(P),\deg(Q))=(72,108)$ we couldn't solve the corresponding system of polynomial equations, thus it is left open.

In section~\ref{seccion 5} we provide another
proof of the case $(\deg(P),\deg(Q))=(56,84)$, without using the machinery of intersection numbers.

We want to emphasize that all the techniques developed in each of the articles~\cite{GGV2}, \cite{GGV3} and~\cite{GGV5}
are very useful in discarding some families of possible counterexamples. However, the Jacobian conjecture is a very hard problem,
so we have to combine these techniques in order to be able to increase the lower bound of 100 for $\max\{\deg(P),\deg(Q)\}$
established in 1983 by Moh in~\cite{M}, up to 108. With enough computing power we would be able to raise it up from 108 to 125,
since there is only one case left.

Since this article continues the work in~\cite{GGV1}, ~\cite{GGV2}, \cite{GGV3} and~\cite{GGV5}, we will use the notations and conventions established in these articles.

\section{Complete list of small cases}
\label{seccion 1}
Based on the tables obtained in sections~5 and~6 of~\cite{GGV5}, we begin with the study of the cases with $\max\{\deg(P),\deg(Q)\} < 125$.
The aim is to prove the following result:

\begin{theorem}
	If $(P,Q)$ is a counterexample to the Jacobian Conjecture, then we have either $\max\{\deg(P),\deg(Q)\} \geq 125$, or
$(\deg(P),\deg(Q))\in\{(72,108),(108,72)\}$.
\end{theorem}

Following~\cite{GGV5}, we take the smallest members of the families
$F_1,F_2,F_3,F_9,F_{17},F_{22}$ in section 5 of~\cite{GGV5} and three additional cases in the tables of section~6 of~\cite{GGV5}.
The following table presents the 10 resulting cases.
 Five of them have been shown to be impossible in various cited articles:
In~\cite{M} the cases with $\max\{\deg(P),\deg(Q)\}\in \{64,75,84,99\}$ have been considered, but only in the case 64 a complete proof is given.
 This case was also computed by Heitman in~\cite{H}. In~\cite{GGV4} this case and the case $(\deg P,\deg Q)=(80,112)$ have been discarded.
The case $\max\{\deg(P),\deg(Q)\}=75$ of Moh in~\cite{M}, corresponds to two cases in two different families in~\cite{GGV5}*{section 5},
and has been discarded in~\cite{GGV3}*{section 5}.
The case $\max\{\deg(P),\deg(Q)\}=120$ can be discarded by a straightforward argument mentioned in~\cite{GGV2}*{Remark 3.31},
which we will describe in detail in the next section. In that section we will discard also
the case $\max\{\deg(P),\deg(Q)\}=84$ using~\cite{GGV6}*{Theorem 7.3}.
\begin{center}
\ra{1.2}
\begin{tabular}[t]{clcc}
\toprule
$A_0$ & $( m,n)$& $\max\{\deg(P),\deg(Q)\}$ & \text{Discarded?} \\
\midrule
              $(4,12)$&{\color{red}(3,4)}&64 & \cite{GGV4}*{section 3.5},\cite{M}, \cite{H} \\
              $(4,12)$&{\color{red}(5,7)}&112 & \cite{GGV4}*{section 3.5} \\
              $(5,20)$&{\color{red}(2,3)}&75 & \cite{GGV3}*{section 5}, no detail in \cite{M} \\
              $(5,20)$&{\color{red}(3,2)}&75 & \cite{GGV3}*{section 5}, no detail in \cite{M} \\
              $(7,21)$&(2,3) & 84 & no detail in \cite{M} \\
              $(8,24)$&{\color{red}(2,3)}& 96 & \cite{GGV5}*{Proposition 6.1}\\
              $(8,28)$&*(3,2)& 108 & -\\
              $(8,32)$&(3,2)&120 & - \\
              $(9,24)$&(2,3) &99 & no detail in \cite{M} \\
              $(9,27)$&(2,3)&108 & -\\
\bottomrule
\end{tabular}
\end{center}
We highlight with red the cases that we consider solved. Four of the remaining five cases will be discarded in this paper. In section~\ref{seccion 2}
we will discard the cases $\max\{\deg(P),\deg(Q)\}=84$ and $\max\{\deg(P),\deg(Q)\}=120$. The case $\max\{\deg(P),\deg(Q)\}=84$
uses one of the main results of~\cite{GGV6}. In section~\ref{seccion 5} we offer another proof for this case.

\section{The cases $\max\{\deg(P),\deg(Q)\}=84$ and $\max\{\deg(P),\deg(Q)\}=120$}
\label{seccion 2}
We first discard the case $\max\{\deg(P),\deg(Q)\}=120$. This is the first case in the first table at page 28 of~\cite{GGV5}.

\begin{figure}[htb]
\centering
\begin{tikzpicture}[scale=0.6]
\draw[step=0.25cm,gray,very thin] (-0.03,-0.03) grid (3.3,8.3);
\draw [->] (0,0)--(8.5,0) node[anchor=north]{$x$};
\draw [->] (0,-0.3) --(0,9) node[anchor=east]{$y$};
\draw [red,thick,-] (1,6) -- (2,8)--(2,7)--(0.25,0);
\draw[blue] (2.5,8.5) node[fill=white, above=0pt, right=0pt]{ $A_0=(8,32)$};
\draw[blue](2.5,6.5) node[fill=white, above=0pt, right=0pt]{ $A_1=(8,28)$};

\fill[red]
(2,8) circle (3pt)
(2,7) circle (3pt)
(0.25,0) circle (3pt);

\draw (9.5,6.5) node[fill=white,right=2pt]{\Large $y\mapsto y+\lambda_1$};

\draw (9.5,5.5) node[fill=white,right=2pt]{\Large $x\mapsto x$};
\draw[->] (8,4) .. controls (11,5) .. (14,4);

\draw[step=0.25cm,gray,very thin] (15.97,-0.03) grid (19.3,8.3);
\draw [->] (16,0)--(24.5,0) node[anchor=north]{$x$};
\draw [->] (16,-0.3) --(16,9) node[anchor=east]{$y$};
\draw [red,thick,-] (17,6) -- (18,8)--(18,1)--(17.5,0.66);
\draw[blue] (18.5,7.5) node[fill=white, above=0pt, right=0pt]{ $A_0=(8,32)$};
\draw[blue](18.5,1.5) node[fill=white, above=0pt, right=0pt]{ $A_0'=(8,4)$};

\fill[red]
(18,8) circle (3pt)
(18,1) circle (3pt);
\end{tikzpicture}
\caption{Discarding $(8,32)$: $(8,4)$ is not a last possible corner}
\end{figure}
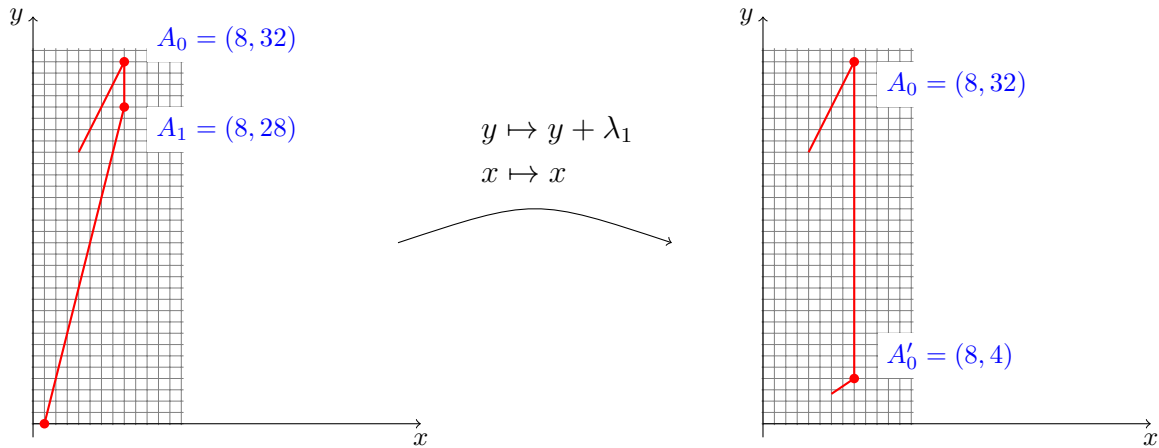

Here $A_0=(8,32)$, $A_1=(8,28)$ and $A_2=(11/4,7)$. Since
$v_{\rho_1,\sigma_1}(A_1)=v_{\rho_1,\sigma_1}(A_2)$ we have $(\rho_1,\sigma_1)=(4,-1)$, and so $\en_{4,-1}(F_1)=(6,21)$ (Note that $(6,21)$
is the only point of the form $\mu(8,28)$ with $v_{\rho_1,\sigma_1}(F_1)=v_{\rho_1,\sigma_1}(A_0)(1,1)=3$).  Since
$$
\en_{4,-1}(F_1)=\frac {p_1}{q_1} A_1=\frac 34 A_1
$$
we obtain that $q_1=4$, since by~\cite{GGV1}*{Theorem 7.6(5)} we know that $q_1=4|d_0$, where $d_0=\max\{d: \ell_{1,0}(P)=R^{md}\}$ (Note that
we use the order of~\cite{GGV5} on the directions, opposed to the order in~\cite{GGV1}). But $d_0\le 4$, since
$d_0\st_{1,0}(R)=\frac{1}{m}\st_{1,0}(P)=(8,28)=4(2,7)$ and $\gcd(2,7)=1$. Hence,
$$
\ell_{1,0}(P)=R^{4m}=(\lambda_P x^2 y^{7}(y-\lambda_1))^{4m},
$$
and if we apply the automorphism $\varphi$ given by $\varphi(x)= x$ and $\varphi(y)=y+\lambda_1$, then
we have $\ell_{1,0}(\varphi(P))= (\lambda_P x^2 y (y+\lambda_1)^{7})^{4m}$ and so $A_0'=(8,4)$ would be the last possible corner
of $\varphi(P)$ (See~\cite{GGV2}*{Definition 3.21}). But this is impossible
by~\cite{GGV2}*{Proposition 3.29}, which discards this case.

Now we discard the case $\max\{\deg(P),\deg(Q)\}=84$. We will use~\cite{GGV6}*{Theorem 7.3}. The present proof is very short, but uses
the machinery of~\cite{GGV6}. In the last section we will give a proof that is very similar to the proofs in the other cases.

We first compute $I_M$.
Remember that we are in the smallest case of the family $F_9$ in the table at page 25 of~\cite{GGV5}. So
$$
 A_0=(7,21),\quad A_0'= (1,0),\quad A_1= (11/7,2),\quad k=1,\quad m= j+2\quad\text{and}\quad n= 2j+3. \qquad
$$

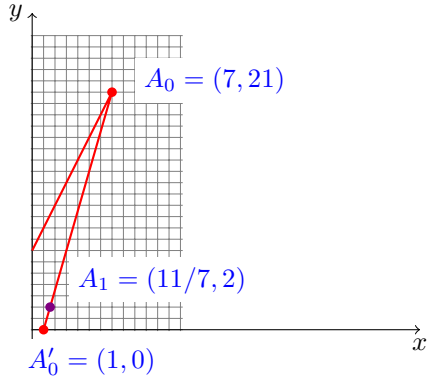
\begin{figure}[htb]
\centering
\begin{tikzpicture}[scale=0.6]
\draw[step=0.25cm,gray,very thin] (-0.03,-0.03) grid (3.3,6.5);
\draw [->] (0,0)--(8.5,0) node[anchor=north]{$x$};
\draw [->] (0,-0.3) --(0,7) node[anchor=east]{$y$};
\draw [red,thick,-] (0,1.75) -- (1.75,5.25)--(0.25,0);
\draw[blue] (2.25,5.5) node[fill=white, above=0pt, right=0pt]{ $A_0=(7,21)$};
\draw[blue](0.8,1.1) node[fill=white, above=0pt, right=0pt]{ $A_1=(11/7,2)$};
\draw[blue](-0.3,-0.7) node[fill=white, above=0pt, right=0pt]{ $A_0'=(1,0)$};

\fill[red]
(1.75,5.25) circle (3pt)
(0.25,0) circle (3pt);

\fill[violet]
(0.3929,0.5) circle (3pt);

\end{tikzpicture}
\caption{The case $\max\{\deg(P),\deg(Q)\}=84$: Family $F_9$.}
\end{figure}

Since $\en_{7,-2}(F)=(5,15)$, we have $F=xy \ov{f}(z^{7})$ for some $\ov{f}$ of degree 2, and so,
by~\cite{GGV1}*{Proposition 2.11(3)}, we have $\ell_{7,-2}(P)=x \ov{p}(z^7)$, where $\ov{p}$
is of degree $3m$ and has at most two different factors, with each factor with a multiplicity that is a multiple of $m$. Since $A_1=(11/7,2)$,
it has exactly two factors. Hence
$\ell_{7,-2}(P)=((z^7-\alpha_1)^2(z^7-\alpha_2))^m$, and so $\ell_{7,-2}(P)$ has 14 different factors. For 7 of them there is a final corner at
$A_1=(11/7,2)=(a/l,b)$,
and these yield the 7 major final $\pi$-roots of $P$. Moreover, for each of them, by~\cite{GGV6}*{Proposition 6.6(2)}, we
have $|D_{\tau}^{P_{\xi}}|=mb=2m$ and $\lambda_{\tau}^{Q}=1/7$, since $k=1$ and $l=7$.
Hence
$$
I_M=\sum_{\tau\in P_M} |D_{\tau}^{P_{\xi}}|\lambda_{\tau}^{Q}= 7 \cdot (2m)\frac 17=2(j+2).
$$
Since we consider the case $j=0$ we obtain $I_M=4$.

The other 7 factors yield seven final minor roots at $A_1=(9/7,1)$, and for each of them we have $\delta_{\tau}=-9/7$. Since $q_0=a=7$, a
standard argument (see~\cite{GGV6}*{Example 6.7})
 shows that there is only one final minor $\pi$-root $\tau$ corresponding to the 7m top minor roots of $P_{\xi}$
(see~\cite{GGV6}*{Proposition 6.4}). Moreover, one verifies that there exist
$\lambda_0$ and $\lambda_{-2}$ such that
$$
\tau=-x+\lambda_0+\lambda_{-2}x^{-2}+\pi x^{-3},
$$
and so $\delta_{\tau} = -3$. Consequently, by definition of $I_m$ (see~\cite{GGV6}*{Theorem 7.3}) we have
$$
I_m=1-\sum_{\tau\in P_m} (\delta_{\tau}+1)=1-(-3+1)-7(-9/7+1)=5.
$$
This contradicts the fact that, by~\cite{GGV6}*{Theorem 7.3}, we have $I_M\ge I_m$, and discards this case.

\section{Reducing the size of the Newton polygon}
\label{seccion 3}

In this section we apply some automorphisms reminiscent of the procedure in section 8 of the ArXiv version of~\cite{GGV1} (arXiv:1401.1784)
 to the Newton Polygons in order to greatly reduce their sizes.

\begin{proposition}[Case (9,27)]  \label{case927}
	If there is a counterexample to the Jacobian Conjecture in the case $(9,27)$, then there exist $P,Q \in L^{(1)}$ with $[P,Q] = x$ and
	\begin{align*}
	N(P) = \{(0,0),(1,1),(6,16), (6,18), (0,18)\} \\
	N(Q) = \{(0,0),(1,0),(9,24), (9,27), (0,27)\}
	\end{align*}
\end{proposition}
\begin{proof}
	The corners of the polygons of $P$ and $Q$ are $\{(0,0),(1,0),(9,24),(9,27),(0,9)\}$ multiplied by $(m,n) = (2,3)$ respectively. The edge $\{(9,27), (0,9) \}$ is given by $y^9(xy^2 - \alpha_1)^9$, because $\enF = \frac{1}{9}(9,27)$ when looking at its corresponding direction, by ~\cite{GGV1}*{Corollary 7.4}. After applying the automorphism $\phi_1$ with $\phi_1(x) = y$ and $\phi_1(y) = x$ and then the automorphism $\phi_2$ with $\phi_2(x) = x$ and $\phi_2(y) = y + \alpha_1 x^{-2}$ we transform the corners of the polygons to $\{(0,0),(27,9),(24,9),(0,1),(-2,0)\}$, again multiplied by $(2,3)$ respectively. \\

\begin{figure}[htb]
\centering
\begin{tikzpicture}[scale=0.4]
\draw[step=0.5cm,gray,very thin] (-0.03,-0.03) grid (5.3,15.3);
\draw [->] (0,0)--(6.5,0) node[anchor=north]{$x$};
\draw [->] (0,-0.3) --(0,17) node[anchor=east]{$y$};
\draw [red,thick,-] (0,4.5) -- (4.5,13.5)--(4.5,12)--(0.5,0);
\draw[blue] (5,13.5) node[fill=white, above=0pt, right=0pt]{ $A_0=(9,27)$};
\draw[blue](5,12) node[fill=white, above=0pt, right=0pt]{ $A_1=(9,24)$};

\fill[red] (0,4.5) circle (3pt)
(4.5,13.5) circle (3pt)
(4.5,12) circle (3pt)
(0.5,0) circle (3pt);

\draw (6,7) node[fill=white,right=2pt]{\Large $y\mapsto x$};

\draw (6,8) node[fill=white,right=2pt]{\Large $x\mapsto y$};
\draw[->] (6,9) .. controls (9,9) .. (11.5,10);

\draw[step=0.5cm,gray,very thin] (12-0.03,10-0.03) grid (27.3,15.3);
\draw [->] (12,10)--(28.5,10) node[anchor=north]{$x$};
\draw [->] (12,10-0.3) --(12,17) node[anchor=east]{$y$};
\draw [red,thick,-] (12,10.5) -- (24,14.5)--(25.5,14.5)--(16.5,10);

\fill[red] (12,10.5) circle (3pt)
(24,14.5) circle (3pt)
(25.5,14.5) circle (3pt)
(16.5,10) circle (3pt);

\draw[step=0.5cm,gray,very thin] (10.97,-0.03) grid (27.3,5.3);
\draw [->] (10,0)--(28.5,0) node[anchor=north]{$x$};
\draw [->] (12,-0.3) --(12,7) node[anchor=east]{$y$};
\draw [red,thick,-] (11,0) -- (12,0.5)--(24,4.5)--(25.5,4.5)--(12,0);

\draw[->] (18,9) .. controls (18.5,7.5) .. (18,6);

\draw (19,7) node[fill=white,right=2pt]{\Large $y\mapsto y+\alpha_1 x^{-2}$};

\draw (19,8) node[fill=white,right=2pt]{\Large $x\mapsto x$};

\fill[red] (11,0) circle (3pt)
(12,0.5) circle (3pt)
(24,4.5) circle (3pt)
(25.5,4.5) circle (3pt)
(12,0) circle (3pt);
\end{tikzpicture}
\caption{Applying $\phi_2\circ \phi_1$ to the case $(9,27)$}
\end{figure}
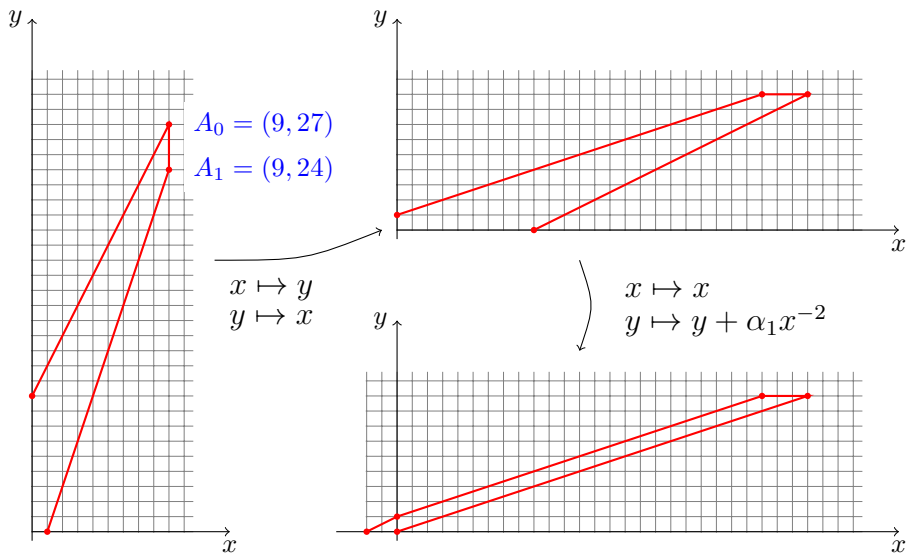
	
	The edge $\{(24,9), (0,1)\}$ is given by $y(yx^3 - \alpha_2)^8$ for some constant $\alpha_2$ (this corresponds to the edge $\{(1,0), (9,24)\}$ in the original polygon which is of this form). Apply the automorphism $\phi_3$ given by $\phi_3(x) = x$ and $\phi_3(y) = y + \alpha_2 x^{-3}$ to reduce this edge to $\{(24,9), (21,8)\}$. Let us analyze the possibilities for the opposite vertex in the other edge containing $(21,8)$.
	
	To do this, set $(\rho_2, \sigma_2) = \min\{\Succ_P(-1,3), \Succ_Q(-1,3)\}$ Then if $\en_{\rho_2,\sigma_2}(P) \sim \en_{\rho_2,\sigma_2}(Q)$, the point $(a',b') = \frac{1}{2} \en_{\rho_2,\sigma_2}(P) = \frac{1}{3} \en_{\rho_2,\sigma_2}(Q)$ could be at any of \\
$\{(-2,0),(-1,0),(1,1),(2,1),(4,2),(5,2),(7,3),(10,4),(13,5)\}$. To discard all but $(5,2)$ and $(13,5)$, it is enough to check that there cannot exist an element $F$ in the other cases. The corner of $F$ which is not $(1,1)$ must be of the form $(1,1) + c(21-a',8-b')/\gcd(21-a',8-b')$ for some positive integer $c$, so that
\begin{align*}
	(1,1) + c\frac{(21-a', 8-b')}{\gcd(21-a',8-b')} = \frac{p}{q} (21,8).
\end{align*}
Computing $v_{-8,21}$ on both sides of this equality gives $13\gcd(21-a',8-b') + c(8a' - 21b') = 0$, which implies that $21b' - 8a' | 13 \gcd(21-a',8-b')$. For each of the possibilities for $(a',b')$ above, we can discard $(1,1)$ as $(a',b')$ cannot lie in the diagonal, and for the rest, only $(5,2)$ and $(13,5)$ satisfy this divisibility condition, as can be seen in the table below.
\begin{center}
\ra{1.2}
\begin{tabular}[t]{clcc}
\toprule
$(a',b')$ & $21b' - 8a'$ & $13 \gcd(21-a', 8-b')$ \\
\midrule
$(-2,0)$ & $16$ & $13$  \\
$(-1,0)$ & $8$ & $26$ \\
$(2,1)$ & $5$ & $13$ \\
$(4,2)$ & $10$ & $13$ \\
$(5,2)$ & $2$ & $26$ \\
$(7,3)$ & $7$ & $13$ \\
$(10,4)$ & $4$ & $13$ \\
$(13,5)$ & $1$ & $13$ \\
\bottomrule
\end{tabular}
\end{center}

	Whether we continue assuming $(a',b') \in \{ (5,2), (13,5) \}$ or we consider instead the case $\en_{\rho_2,\sigma_2}(P) \not \sim \en_{\rho_2,\sigma_2}(Q)$, we can apply \cite{GGV1}*{Proposition 8.2} and get the existence of $k \in \mathbb{N}$ with
	\[
		(k+1)b < a \hspace{1cm} \text{ and } \hspace{1cm} \{\en_{\rho,\sigma}(P), \en_{\rho,\sigma}(Q)\} = \{(-k,0), (k+1,1) \},
	\]

\begin{figure}[htb]
\centering
\begin{tikzpicture}[scale=0.2]
\draw[step=0.5cm,gray,very thin] (-0.03,17-0.03) grid (28.3,27.3);
\draw [thick,->] (0,17)--(29,17) node[anchor=north]{$x$};
\draw [thick,->] (0,17-0.3) --(0,28) node[anchor=east]{$y$};
\draw [red,thick,-] (0,17) -- (1,17.5)--(21,25)--(24,26)--(27,26)--(0,17);

\fill[red] (0,17) circle (5pt)
(1,17.5) circle (5pt)
(21,25) circle (5pt)
(24,26) circle (5pt)
(27,26) circle (5pt);

\draw[step=0.5cm,gray,very thin] (-1.03,-0.03) grid (41.3,13.8);
\draw [thick,->] (0,0)--(42.5,0) node[anchor=north]{$x$};
\draw [thick,->] (0,-0.3) --(0,15) node[anchor=east]{$y$};
\draw [red,thick,-] (-0.5,0) -- (31.5,12)--(36,13.5)--(40.5,13.5)--(0,0);

\fill[red] (-0.5,0) circle (5pt)
(31.5,12) circle (5pt)
(36,13.5) circle (5pt)
(40.5,13.5) circle (5pt)
(0,0) circle (5pt);

\draw (19,21) node[fill=white,right=4pt]{\Huge $P$};

\draw (29,7) node[fill=white,right=4pt]{\Huge $Q$};

\end{tikzpicture}
\caption{The shape of $P$ and $Q$}
\end{figure}
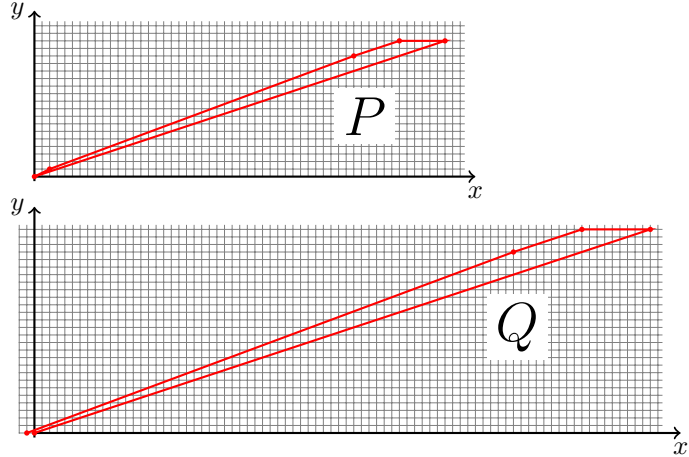

	where $(a,b)$ is one of $\{(21,8), (13,5), (5,2)\}$ and $(\rho,\sigma)$ is the direction corresponding to the edge in question.
In all cases this gives $k = 1$ and so  $\{\en_{\rho,\sigma}(P), \en_{\rho,\sigma}(Q)\}$ $= \{(-1,0), (2,1) \}$, with the same direction
$(\rho, \sigma) = (-3,8)$. (Note that if we started assuming $(a',b') \in \{(5,2), (13,5)\}$ then we have reached a contradiction, as we have a
different end for the direction $(-3,8)$.) Since $\st_{-3,8}(P) = (42,16)$ and $\st_{-3,8}(Q) = (63,24)$, we get that $\en_{-3,8}(P) = (2,1)$ and
$\en_{-3,8}(Q) = (-1,0)$. In fact, $(-3,8) \times ((42,16)-(-1,0)) \neq 0$.
	
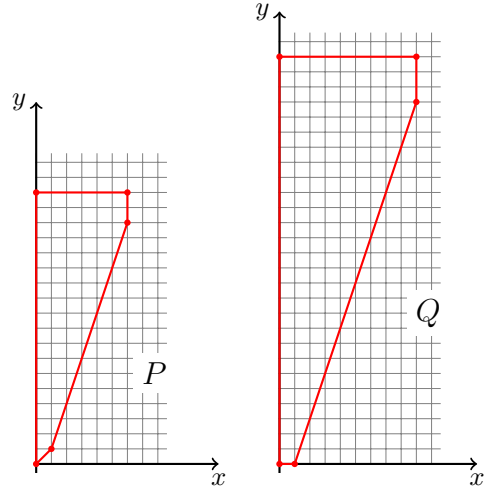
\begin{figure}[htb]
\centering
\begin{tikzpicture}[scale=0.4]
\draw[step=0.5cm,gray,very thin] (-0.03,-0.03) grid (4.3,10.3);
\draw [thick,->] (0,0)--(6,0) node[anchor=north]{$x$};
\draw [thick,->] (0,-0.3) --(0,12) node[anchor=east]{$y$};
\draw [red,thick,-] (0,0) -- (0.5,0.5)--(3,8)--(3,9)--(0,9)--(0,0);

\fill[red] (0,0) circle (3pt)
(0.5,0.5) circle (3pt)
(3,8) circle (3pt)
(3,9) circle (3pt)
(0,9) circle (3pt);

\draw[step=0.5cm,gray,very thin] (7.97,-0.03) grid (13.3,14.3);
\draw [thick,->] (8,0)--(14.5,0) node[anchor=north]{$x$};
\draw [thick,->] (8,-0.3) --(8,15) node[anchor=east]{$y$};
\draw [red,thick,-] (8,0) -- (8.5,0)--(12.5,12)--(12.5,13.5)--(8,13.5) --(8,0);

\fill[red] (8,0) circle (3pt)
(8.5,0) circle (3pt)
(12.5,12) circle (3pt)
(12.5,13.5) circle (3pt)
(8,13.5) circle (3pt);

\draw (3,3) node[fill=white,right=2pt]{\Large $P$};

\draw (12,5) node[fill=white,right=2pt]{\Large $Q$};

\end{tikzpicture}
\caption{The shape of $P$ and $Q$ according to Proposition~\ref{case927}}
\end{figure}

	For convenience, we now apply the morphism $\varphi$ such that $\varphi(x) = x^{-1}$ and $\varphi(y) = x^3y$. Note that this is an
 automorphism of $L^{(1)}=K[x,x^{-1},y]$ but not of $K[x,y]$. By the chain rule we have
$$
[\varphi(P), \varphi(Q)] = \varphi[P,Q] [\varphi(x), \varphi(y)] = -[P,Q] x,
$$
and multiplying by a nonzero constant we can assume $[\varphi(P), \varphi(Q)]=x$. This transforms the polygons of $P$ and $Q$ into
\begin{align*}
	N(P) = \{(0,0),(1,1),(6,16), (6,18), (0,18)\} \\
	N(Q) = \{(0,0),(1,0),(9,24), (9,27), (0,27)\}
\end{align*}
as desired.
\end{proof}

\begin{proposition}[Case (9,24)] \label{case924}
	If there is a counterexample to the Jacobian Conjecture in the case $(9,24)$, then there exist $P,Q \in L^{(1)}$ with $[P,Q] = x$ and one of the following cases holds:
	\begin{enumerate}
	\item
	$N(P) = \{(0,0),(1,1),(6,16), (6,18), (0,12)\} , \;
	N(Q) = \{(0,0),(1,0),(9,24), (9,27), (0,18)\}$
	\item  	$N(P) = \{(0,0),(1,1),(6,16), (6,18), (0,6)\},\phantom{2} \;
	N(Q) = \{(0,0),(1,0),(9,24), (9,27), (0,9)\}$
	\item $	N(P) = \{(0,0),(1,1),(6,16), (6,18)\},\phantom{, (0,12)} \;
	N(Q) = \{(0,0),(1,0),(9,24), (9,27)\}$
	\end{enumerate}

\end{proposition}
\begin{proof}
	We will prove first that the corners of the polygons of $P$ and $Q$ are $\{(0,0),(1,0),(9,24), (0,6)\}$ multiplied by $(m,n) = (2,3)$.

For this we apply the automorphism $\phi_1$ with $\phi_1(x) = y$, $\phi_1(y) = x$ and we use ~\cite{GGV1}*{Corollary 7.4} with
$(\rho_0,\sigma_0)=(-1,3)$, $l=1$ and $(a/l,b)=(24,9)$. Then $\en_{-1,3}(F) = (16,6)=\frac{2}{3}(24,9)$, which means that $q=3$ in the corollary,
 and for
$(\rho,\sigma)=\Pred_P(-1,3)$ we have $(\rho,\sigma)\in ](0,-1),(1,-1)]$ and there exists a $(\rho,\sigma)$-homogeneous element $R$,
such that $\ell_{\rho,\sigma}(P)=\lambda R^{qm}=\lambda R^{3m}$ and $v_{\rho,\sigma}(R)>0$. Thus $\en_{\rho,\sigma}(R)=(8,3)$, and there exists a
$(\rho,\sigma)$-homogeneous
element $G$ satisfying
\begin{equation}\label{condicion}
  [R,G]=R^i\quad\text{for some}\quad i.
\end{equation}
 Therefore, we are in the setting of~\cite{GGV2}*{Proposition 3.12}. The following
algorithm with $l=1$, $a=8$ and $b=3$, shows that $\Pred_{P}(1,0)\in\{(1,-2),(2,-5)\}$.

\begin{algorithm}[H]\label{algoritmo}
    \SetAlgoLined\DontPrintSemicolon
    \SetKwFunction{PossibleStartingPoints}{PossibleStartingPoints}
    \SetKwProg{myalg}{Algorithm}{}{}
    \SetKwProg{myproc}{Procedure}{}{}
    \myproc{\PossibleStartingPoints{$M$}}{
    \KwIn{A corner $(a/l,b)\in\frac 1l \mathds{N}\times \mathds{N}$}
	\KwOut{A list $\PossibleStartingPoints$ of points $(c/l,d)$ such that~\eqref{condicion} is satisfied}
        \For{$d=0$ \KwTo $b-1$}{
            \For{$c=\lfloor d*\frac ab\rfloor+1$ \KwTo $ld+a-bl-1$}{
                $N_1 = \gcd(a - c, b - d)$;\\
                $N_2 = \gcd(c, d)$;\\
                $(\rho,\sigma) =\dir((a/l,b)-(c/l,d))$;\\
                $s = (\rho a+\sigma l b)/\gcd(\rho a+\sigma l b, l(\rho+\sigma))$;\\
                \If{($s\mid N_2$ \normalfont\textbf{ and } $d>0$), \normalfont\textbf{ or } $s\le N_1$}
                    {$(c/l,d) \hookrightarrow \PossibleStartingPoints$}
                }
            }
    }
    {\bf RETURN} List $\PossibleStartingPoints$
    \caption{Possible starting points}
\end{algorithm}

 But $\Pred_{P}(1,0)=(2,-5)$ leads to $\st_{2,-5}(P)=3m(3,1)$,
and $\Pred_{P}(2,-5)\le (1,-3)$, and then $\deg_x(P(x,0))\le 0$,  which contradicts~\cite{vdE}*{Proposition 10.2.6}, since $P\in kK[x,y]$.
Similarly $\st_{1,-2}(R)=(6,2)$ leads to $\Pred_{P}(1,-2)\le (1,-3)$, using the above algorithm. But then $\deg_x(P(x,0))\le 0$,  which again
contradicts~\cite{vdE}*{Proposition 10.2.6}. By the same~\cite{vdE}*{Proposition 10.2.6} we have $\Pred_{P}(1,0)\le (3,-8)$ and
$\st_{\rho,\sigma}(R)\ne (0,0)$ for $(\rho,\sigma)=\Pred_{P}(1,0)$.

Thus $\Pred_{P}(1,0)=(1,-2)$ and $\st_{1,-2}(R)\in \{(2,0),(4,1)\}$.

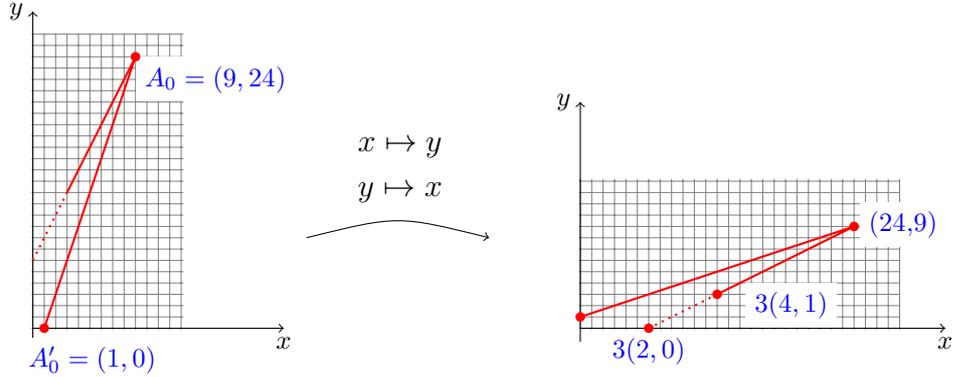
\begin{figure}[htb]
\centering
\begin{tikzpicture}[scale=0.6]
\draw[step=0.25cm,gray,very thin] (-0.03,-0.03) grid (3.3,6.5);
\draw [->] (0,0)--(5.5,0) node[anchor=north]{$x$};
\draw [->] (0,-0.3) --(0,7) node[anchor=east]{$y$};
\draw [red,thick,dotted] (0,1.5) -- (0.75,3);
\draw [red,thick,-] (0.75,3) -- (2.25,6)--(0.25,0);
\draw[blue] (2.25,5.5) node[fill=white, above=0pt, right=0pt]{ $A_0=(9,24)$};
\draw[blue](-0.3,-0.7) node[fill=white, above=0pt, right=0pt]{ $A_0'=(1,0)$};
\fill[red]
(2.25,6) circle (3pt)
(0.25,0) circle (3pt);
\draw (6.8,3) node[fill=white,right=2pt]{\Large $y\mapsto x$};

\draw (6.8,4) node[fill=white,right=2pt]{\Large $x\mapsto y$};
\draw[->] (6,2) .. controls (8,2.5) .. (10,2);

\draw[step=0.25cm,gray,very thin] (12-0.03,-0.03) grid (19,3.3);
\draw [->] (12,0)--(20,0) node[anchor=north]{$x$};
\draw [->] (12,-0.3) --(12,5) node[anchor=east]{$y$};
\draw [red,thick,dotted] (15,0.75)--(13.5,0);
\draw [red,thick,-] (12,0.25) -- (18,2.25)--(15,0.75);
\fill[red] (12,0.25) circle (3pt)
(15,0.75) circle (3pt)
(18,2.25) circle (3pt)
(13.5,0) circle (3pt);
\draw[blue] (13.5,-0.5) node {$3(2,0)$};
\draw[blue] (15.5,0.5) node[fill=white,right=2pt] {$3(4,1)$};
\draw[blue] (18,2.3) node[fill=white,right=2pt]{(24,9)};

\end{tikzpicture}
\caption{$\Pred_{P}(1,0)=(-1,2)$ and $\st_{1,-2}(R)\in \{(2,0),(4,1)\}$.}
\end{figure}

By~\cite{GGV1}*{Corollary 7.4} we have
$\ell_{1,-2}(P_1)=\lambda R^{3m}$, and $R$ has either one, two or three linear factors in $z=x^2y$, hence
$$
R=x^2(z-\lambda_1)^3,\quad R=x^2(z-\lambda_1)^2(z-\lambda_2)\quad \text{or}\quad R=x^2(z-\lambda_1)(z-\lambda_2)(z-\lambda_3),
$$
for three different $\lambda_i$.
Let $\ov\varphi$ be given by $\ov\varphi(x)=x$ and $\ov\varphi(y)=y+\lambda_1x^{-2}$.
In the last case we obtain $(c,d)=(12,3)$ and $(a,b)=(24,9)$, and $s=\vartheta,N_1,N_2$ in~\cite{GGV2}*{Proposition 3.12} are given by
$s=\vartheta=6$, $N_1=6$, $N_2=3$. By~\cite{GGV2}*{Proposition 3.12}(2) there is a linear factor of $R^3$ with multiplicity $s=\vartheta=6$,
contradicting the fact that all linear factors of $R$ are different, thus eliminating the third case.

Moreover $\lambda_1\ne 0$, since in the first case $\lambda_1=0$ implies $\Pred_P(1,0)=(2,-5)$ or
$\Pred_P(1,0)\le (8,-3)$, which is impossible by the argument above; and similarly in the second case $\lambda_1=0$ implies
$\st_{\rho,\sigma}(R)=(6,2)$, which is also impossible by the argument above.

If $R=x^2(z-\lambda_1)^2(z-\lambda_2)$, apply an automorphism $\phi_2$ with $\phi_2(x) = x$ and $\phi_2(y) = y + \lambda_1 x^{-2}$ and we transform the corners of the polygons to $\{(0,0),(18,6),(24,9),(0,1),(-2,0)\}$, again multiplied by $(2,3)$ respectively. In fact,
if $\Pred_P(1,-2)=(\rho,\sigma)$, then
$\ell_{\rho,\sigma}(P)=\lambda R_1^{3m}$, hence $\en_{\rho,\sigma}(R_1)=(6,2)$, and the algorithm above shows that
$\st_{\rho,\sigma}(R_1)=(0,0)$

After doing the transformations on the edge $\{(24,9), (0,1)\}$ exactly as in the case $(9,27)$, one obtains the desired form:
\begin{align*}
	N(P) = \{(0,0),(1,1),(6,16), (6,18), (0,12)\} \\
	N(Q) = \{(0,0),(1,0),(9,24), (9,27), (0,18)\}.
\end{align*}

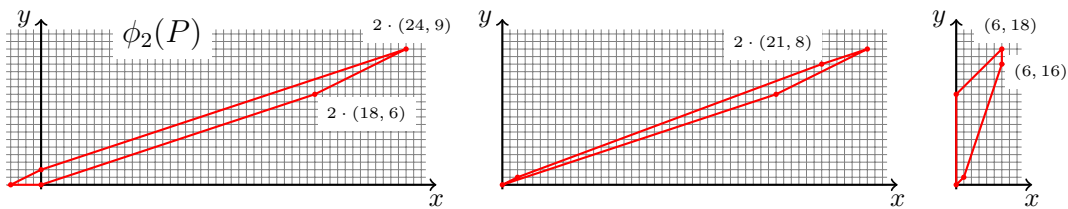
\begin{figure}[htb]
\centering
\begin{tikzpicture}[scale=0.2]
\draw[step=0.5cm,gray,very thin] (-2.3,17-0.03) grid (25.3,27.3);
\draw [thick,->] (0,17)--(26,17) node[anchor=north]{$x$};
\draw [thick,->] (0,17-0.3) --(0,28) node[anchor=east]{$y$};
\draw [red,thick,-] (0,17) -- (-2,17)--(0,18)--(24,26)--(18,23)--(0,17);

\fill[red] (0,17) circle (5pt)
(0,18) circle (5pt)
(-2,17) circle (5pt)
(24,26) circle (5pt)
(18,23) circle (5pt);
\draw (4,27) node[fill=white,right=4pt]{\Large $\phi_2(P)$};
\draw (21,27.5) node[fill=white,right=1pt] {\tiny $2\cdot (24,9)$};
\draw (18,21.7) node[fill=white,right=1pt] {\tiny $2\cdot (18,6)$};
\end{tikzpicture}
\begin{tikzpicture}[scale=0.2]
\draw[step=0.5cm,gray,very thin] (-0.03,17-0.03) grid (25.3,27.3);
\draw [thick,->] (0,17)--(26,17) node[anchor=north]{$x$};
\draw [thick,->] (0,17-0.3) --(0,28) node[anchor=east]{$y$};
\draw [red,thick,-] (0,17) -- (1,17.5)--(21,25)--(24,26)--(18,23)--(0,17);

\fill[red] (0,17) circle (5pt)
(1,17.5) circle (5pt)
(21,25) circle (5pt)
(24,26) circle (5pt)
(18,23) circle (5pt);
\draw (14.4,26.4) node[fill=white,right=1pt] {\tiny $2\cdot (21,8)$};

\end{tikzpicture}
\begin{tikzpicture}[scale=0.2]
\draw[step=0.5cm,gray,very thin] (-0.03,17-0.03) grid (4.3,27.3);
\draw [thick,->] (0,17)--(5,17) node[anchor=north]{$x$};
\draw [thick,->] (0,17-0.3) --(0,28) node[anchor=east]{$y$};
\draw [red,thick,-] (0,17) -- (0.5,17.5)--(3,25)--(3,26)--(0,23)--(0,17);
\fill[red] (0,17) circle (5pt)
(0.5,17.5) circle (5pt)
(3,25) circle (5pt)
(3,26) circle (5pt)
(0,23) circle (5pt);
\draw (1,27.5) node[fill=white,right=1pt] {\tiny $(6,18)$};
\draw (3,24.5) node[fill=white,right=1pt] {\tiny $(6,16)$};
\end{tikzpicture}

\caption{The transformation of $P$ when $R=x^2(z-\lambda_1)^2(z-\lambda_2)$.}
\end{figure}

On the other hand, if $R=x^2(z-\lambda_1)^3$ then the edge $\{(6,0), (24,9)\}$ is completely cut by $\phi_2$, and the new (preceding) edge $\{(a',b'),(24,9)\}$ must satisfy $(a',b') \in \{(0,0), (9,3)\}$. In fact, $\en F = \frac{2}{3}(9,24)$ and so the algorithm above with $l=1$, $a=8$ and $b=3$, shows that $\st_{\rho,\sigma}(\phi_2(P)) \in 3m\{(0,0),(3,1)\}$.
 Transforming the edge $\{(24,9),(0,1)\}$ as before, we arrive at the polygons
\begin{align*}
	N(P) = \{(0,0),(1,1),(6,16), (6,18)\} \\
	N(Q) = \{(0,0),(1,0),(9,24), (9,27)\}
\end{align*}

\begin{figure}[htb]
\centering
\begin{tikzpicture}[scale=0.2]
\draw[step=0.5cm,gray,very thin] (-2.3,17-0.03) grid (25.3,27.3);
\draw [thick,->] (0,17)--(26,17) node[anchor=north]{$x$};
\draw [thick,->] (0,17-0.3) --(0,28) node[anchor=east]{$y$};
\draw [red,thick,-] (0,17) -- (-2,17)--(0,18)--(24,26)--(0,17);

\fill[red] (0,17) circle (5pt)
(0,18) circle (5pt)
(-2,17) circle (5pt)
(24,26) circle (5pt);
\draw (4,27) node[fill=white,right=4pt]{\Large $\phi_2(P)$};
\draw (21,27.5) node[fill=white,right=1pt] {\tiny $2\cdot (24,9)$};
\end{tikzpicture}
\begin{tikzpicture}[scale=0.2]
\draw[step=0.5cm,gray,very thin] (-0.03,17-0.03) grid (25.3,27.3);
\draw [thick,->] (0,17)--(26,17) node[anchor=north]{$x$};
\draw [thick,->] (0,17-0.3) --(0,28) node[anchor=east]{$y$};
\draw [red,-] (0,17) -- (1,17.5)--(21,25)--(24,26)--(0,17);

\fill[red] (0,17) circle (5pt)
(1,17.5) circle (5pt)
(21,25) circle (5pt)
(24,26) circle (5pt);
\draw (14.4,26.4) node[fill=white,right=1pt] {\tiny $2\cdot (21,8)$};

\end{tikzpicture}
\begin{tikzpicture}[scale=0.2]
\draw[step=0.5cm,gray,very thin] (-0.03,17-0.03) grid (4.3,27.3);
\draw [thick,->] (0,17)--(5,17) node[anchor=north]{$x$};
\draw [thick,->] (0,17-0.3) --(0,28) node[anchor=east]{$y$};
\draw [red,thick,-] (0,17) -- (0.5,17.5)--(3,25)--(3,26)--(0,17);
\fill[red] (0,17) circle (5pt)
(0.5,17.5) circle (5pt)
(3,25) circle (5pt)
(3,26) circle (5pt);
\draw (1,27.5) node[fill=white,right=1pt] {\tiny $(6,18)$};
\draw (3,24.5) node[fill=white,right=1pt] {\tiny $(6,16)$};
\end{tikzpicture}

\caption{The transformation of $P$ when $R=x^2(z-\lambda_1)^3$ and $(a',b')=(0,0)$.}
\end{figure}
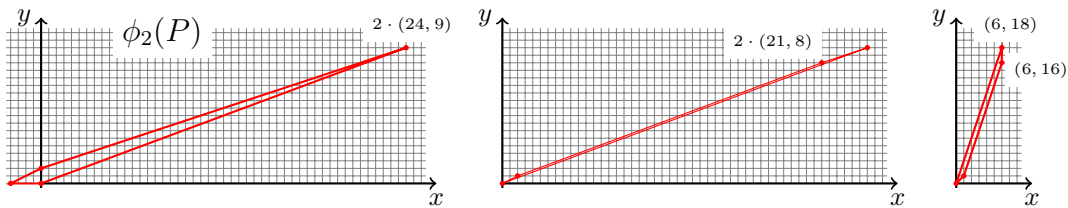

and
 \begin{align*}
	N(P) = \{(0,0),(1,1),(6,16), (6,18), (0,6)\} \\
	N(Q) = \{(0,0),(1,0),(9,24), (9,27), (0,9)\}
\end{align*}

respectively.

\begin{figure}[htb]
\centering
\begin{tikzpicture}[scale=0.2]
\draw[step=0.5cm,gray,very thin] (-2.3,17-0.03) grid (25.3,27.3);
\draw [thick,->] (0,17)--(26,17) node[anchor=north]{$x$};
\draw [thick,->] (0,17-0.3) --(0,28) node[anchor=east]{$y$};
\draw [red,thick,-] (0,17) -- (-2,17)--(0,18)--(24,26)--(9,20)--(0,17);

\fill[red] (0,17) circle (5pt)
(0,18) circle (5pt)
(-2,17) circle (5pt)
(24,26) circle (5pt)
(9,20) circle (5pt);
\draw (4,27) node[fill=white,right=4pt]{\Large $\phi_2(P)$};
\draw (21,27.5) node[fill=white,right=1pt] {\tiny $2\cdot (24,9)$};
\draw (9,19) node[fill=white,right=1pt] {\tiny $2\cdot (9,3)$};
\end{tikzpicture}
\begin{tikzpicture}[scale=0.2]
\draw[step=0.5cm,gray,very thin] (-0.03,17-0.03) grid (25.3,27.3);
\draw [thick,->] (0,17)--(26,17) node[anchor=north]{$x$};
\draw [thick,->] (0,17-0.3) --(0,28) node[anchor=east]{$y$};
\draw [red,thick,-] (0,17) -- (1,17.5)--(21,25)--(24,26)--(9,20)--(0,17);

\fill[red] (0,17) circle (5pt)
(1,17.5) circle (5pt)
(21,25) circle (5pt)
(24,26) circle (5pt)
(9,20) circle (5pt);
\draw (14.4,26.4) node[fill=white,right=1pt] {\tiny $2\cdot (21,8)$};

\end{tikzpicture}
\begin{tikzpicture}[scale=0.2]
\draw[step=0.5cm,gray,very thin] (-0.03,17-0.03) grid (4.3,27.3);
\draw [thick,->] (0,17)--(5,17) node[anchor=north]{$x$};
\draw [thick,->] (0,17-0.3) --(0,28) node[anchor=east]{$y$};
\draw [red,thick,-] (0,17) -- (0.5,17.5)--(3,25)--(3,26)--(0,20)--(0,17);
\fill[red] (0,17) circle (5pt)
(0.5,17.5) circle (5pt)
(3,25) circle (5pt)
(3,26) circle (5pt)
(0,20) circle (5pt);
\draw (1,27.5) node[fill=white,right=1pt] {\tiny $(6,18)$};
\draw (3,24.5) node[fill=white,right=1pt] {\tiny $(6,16)$};
\end{tikzpicture}

\caption{The transformation of $P$ when $R=x^2(z-\lambda_1)^3$ and $(a',b')=(9,3)$.}
\end{figure}
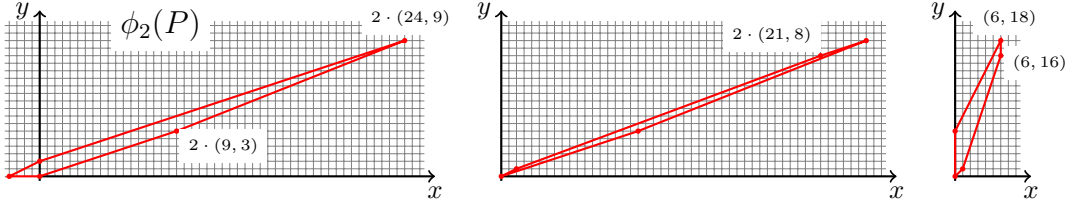

\end{proof}

\begin{proposition}[Case (8,28)]
	If there is a counterexample to the Jacobian Conjecture in the case $(8,28)$, then there exist $P,Q \in L^{(1)}$ with $[P,Q] = x^2$ and one of the following cases holds:
\begin{enumerate}
	\item $N(P) = \{(0,0),(1,0),(8,14), (8,16), (0,8)\}, \; N(Q) = \{(0,0),(2,1),(12,21), (12,24), (0,12)\}.$
	\item $N(P) = \{(0,0),(1,0),(8,14), (8,16)\}, \;	N(Q) = \{(0,0),(2,1),(12,21), (12,24)\}$
\end{enumerate}
	
\end{proposition}
\begin{proof}
We first prove that the corners of the polygons of $P$ and $Q$ in this case are given by
$\{ (0,0), (1,0), (8,28), (0,4) \}$, multiplied by $(m,n) = (3,2)$.

 For this we apply the automorphism $\phi_1$ with $\phi_1(x) = y$, $\phi_1(y) = x$ and we use ~\cite{GGV1}*{Corollary 7.4} with
$(\rho_0,\sigma_0)=(-1,4)$, $l=1$ and $(a/l,b)=(28,8)$. Then $\en_{-1,4}(F) = (21,6)=\frac{3}{4}(28,8)$, which means that $q=4$ in the corollary,
 and for
$(\rho,\sigma)=\Pred_P(-1,4)$ we have $(\rho,\sigma)\in ](0,-1),(1,-1)]$ and there exists a $(\rho,\sigma)$-homogeneous element $R$,
such that $\ell_{\rho,\sigma}(P)=\lambda R^{qm}=\lambda R^{4m}$ and $v_{\rho,\sigma}(R)>0$. Thus $\en_{\rho,\sigma}(R)=(7,2)$, and there exists a
$(\rho,\sigma)$-homogeneous
element $G$ satisfying  $[R,G]=R^i$ for some $i$.
Therefore, we are in the setting of~\cite{GGV2}*{Proposition 3.12}. By~\cite{GGV6}*{Proposition 2.5}, necessarily
$\Pred_P(1,0)\in\{(1,-2),(1,-3)\}$.
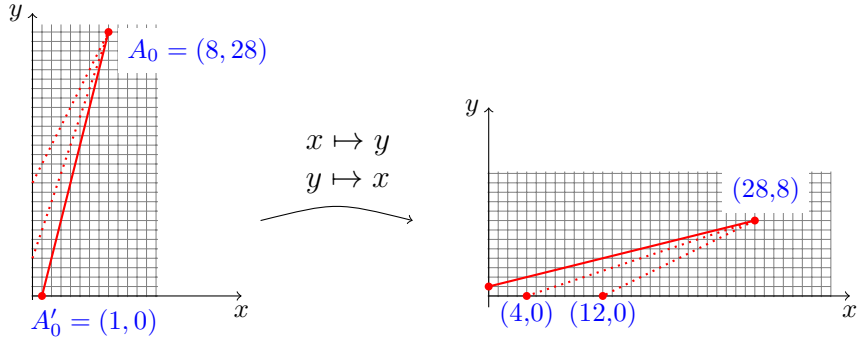
\begin{figure}[htb]
\centering
\begin{tikzpicture}[scale=0.5]
\draw[step=0.25cm,gray,very thin] (-0.03,-0.03) grid (3.3,7.2);
\draw [->] (0,0)--(5.5,0) node[anchor=north]{$x$};
\draw [->] (0,-0.3) --(0,7.5) node[anchor=east]{$y$};
\draw [red,thick,-]  (2,7)--(0.25,0);
\draw [red,thick,dotted] (0,1) -- (2,7);
\draw [red,thick,dotted] (0,3) -- (2,7);

\draw[blue] (2.25,6.5) node[fill=white, above=0pt, right=0pt]{ $A_0=(8,28)$};
\draw[blue](-0.3,-0.7) node[fill=white, above=0pt, right=0pt]{ $A_0'=(1,0)$};
\fill[red]
(2,7) circle (3pt)
(0.25,0) circle (3pt);
\draw (6.8,3) node[fill=white,right=2pt]{\Large $y\mapsto x$};

\draw (6.8,4) node[fill=white,right=2pt]{\Large $x\mapsto y$};
\draw[->] (6,2) .. controls (8,2.5) .. (10,2);

\draw[step=0.25cm,gray,very thin] (12-0.03,-0.03) grid (21,3.3);
\draw [->] (12,0)--(21.5,0) node[anchor=north]{$x$};
\draw [->] (12,-0.3) --(12,5) node[anchor=east]{$y$};
\draw [red,thick,-] (12,0.25) -- (19,2);
\draw [red,thick,dotted] (19,2)--(13,0);
\draw [red,thick,dotted] (19,2)--(15,0);

\fill[red] (12,0.25) circle (3pt)
(19,2) circle (3pt)
(15,0) circle (3pt)
(13,0) circle (3pt);
\draw[blue] (13,-0.5) node {(4,0)};
\draw[blue] (15,-0.5) node {(12,0)};

\draw[blue] (18,2.8) node[fill=white,right=2pt]{(28,8)};
\end{tikzpicture}
\caption{$\Pred_P(1,0)\in\{(1,-2),(1,-3)\}$}
\end{figure}

By the same proposition, if $\Pred_P(1,0)=(1,-2)$, then $\ell_{1,-2}(P)=\lambda (z-\lambda_1)^{8m}$, which implies that
if we apply the morphism given by $\varphi(x)=x$ and  $\varphi(y)=y+\lambda_1 x^{-2}$, then $\Pred_{\varphi(P)}(1,0)\in \{(1,-3),(2,-7)\}$. If
$\Pred{\varphi(P)}(1,0)=(2,-7)$ then the lower side of $\varphi(P)$ has only the corners $m(-2,0)$, $(0,0)$ and $m(28,8)$.

If $\Pred_{\varphi(P)}(1,0)=(1,-3)$, then there are one or two different linear factors in $\ell_{1,-3}(\varphi(P))$.
If there is only one factor, we apply the morphism
 $x\mapsto x$, $y\mapsto y+\lambda x^{-3}$ for adequate $\lambda$ and obtain that the lower side of $P$ has
only the corners $m(-3,0)$, $(0,0)$ and $m(28,8)$.
If are two factors  in $\ell_{1,-3}(P)$, then $\ell_{1,-3}(P)=\lambda x^{4m}(x^3y-\alpha_1)^{4m}(x^3y-\alpha_2)^{4m}$. We apply the morphism
 $x\mapsto x$, $y\mapsto y+\alpha_1 x^{-3}$ and obtain that the lower side of $P$ has
only the corners $m(-3,0)$, $(0,0)$, $m(16,4)$ and $m(28,8)$.

Thus after the transformations we have three possibilities for the Newton polygon of $P$ and $Q$:
\begin{itemize}
  \item[a)] $(m,n)$ times $\{(-2,0),(0,0),(28,8),(0,1)\}$,
  \item[b)] $(m,n)$ times $\{(-3,0),(0,0),(28,8),(0,1)\}$,
  \item[c)] $(m,n)$ times $\{(-3,0),(0,0),(16,4),(28,8),(0,1)\}$.
\end{itemize}

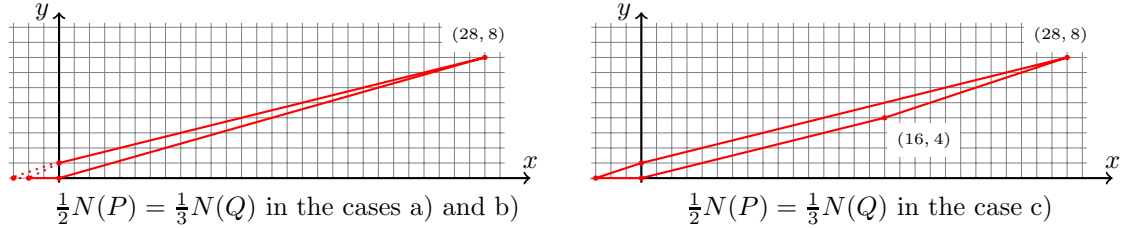
\noindent \begin{tikzpicture}[scale=0.2]
\draw(34,0)node{};
\draw[step=1cm,gray,very thin] (-3.3,-0.03) grid (29.3,10.3);
\draw [thick,->] (0,0)--(31,0) node[anchor=south]{$x$};
\draw [thick,->] (0,-0.3) --(0,11) node[anchor=east]{$y$};
\draw [red,thick,-] (0,1)--(28,8)--(0,0)--(-2,0);
\draw [red,thick,dotted] (0,1)--(-2,0);
\draw [red,thick,dotted] (0,1.1)--(-3.2,0);
\draw [red,thick,dotted] (-2,0)--(-3,0);
\fill[red] (0,0) circle (5pt)
(0,1) circle (5pt)
(-2,0) circle (5pt)
(-3,0) circle (5pt)
(28,8) circle (5pt);
\draw (25,9.5) node[fill=white,right=1pt] {\tiny $(28,8)$};
\draw (15,-2)node{$\frac{1}{2}N(P)=\frac{1}{3}N(Q)$ in the cases a) and b)};
\end{tikzpicture}
\begin{tikzpicture}[scale=0.2]
\draw[step=1cm,gray,very thin] (-3.3,-0.03) grid (29.3,10.3);
\draw [thick,->] (0,0)--(31,0) node[anchor=south]{$x$};
\draw [thick,->] (0,-0.3) --(0,11) node[anchor=east]{$y$};
\draw [red,thick,-] (0,1)--(28,8)--(16,4)--(0,0)--(-3,0)--(0,1);
\fill[red] (0,0) circle (5pt)
(0,1) circle (5pt)
(16,4) circle (5pt)
(-3,0) circle (5pt)
(28,8) circle (5pt);
\draw (25,9.5) node[fill=white,right=1pt] {\tiny $(28,8)$};
\draw (16,2.5) node[fill=white,right=1pt] {\tiny $(16,4)$};
\draw (15,-2)node{$\frac{1}{2}N(P)=\frac{1}{3}N(Q)$ in the case c)};
\end{tikzpicture}

In all the three cases the edge $\{(28,8), (1,0) \}$ must be of the form $y (x^4y - \alpha)^7$, corresponding to its form before the transformations. Apply then the automorphism $\phi_3$ given by $\phi_3(x) = x$, $\phi_3(y) = y + \alpha x^{-4}$, reducing the edge $\{(28,8), (0,1)\}$ to $\{(28,8),(24,7)\}$. As in Proposition ~\ref{case927}, one can analyze the possibilities for the opposite vertex $(a,b)$ in the other edge containing $(24,7)$ and obtain that $\Succ_P(-1,4) = \Succ_Q(-1,4) = (-2,7)$. We can apply \cite{GGV1}*{Proposition 8.2} and get the existence of $k \in \mathbb{N}$ with
	\[
		(k+1)b < a \hspace{1cm} \text{ and } \hspace{1cm} \{\en_{\rho,\sigma}(P), \en_{\rho,\sigma}(Q)\} = \{(-k,0), (k+1,1) \}.
	\]
	where $(a,b)$ is one of $\{(24,7), (17,5), (10,3), (3,1)\}$ and $(\rho,\sigma)$ is the direction corresponding to the edge in question, obtaining $k \in \{1,2\}$. The case $k = 2$ is impossible, as the edges of $P$ and $Q$ would have no way of being parallel. In the case $k = 1$, we can set $(\en_{\rho,\sigma}(Q), \en_{\rho,\sigma}(P))= ((2,1), (-1,0))$.
In the first two cases a) and b) we obtain
\begin{align*}
   &  N(P)=\{(-1,0),(0,0),2(28,8),2(24,7)\}\\
   & N(Q)=\{(2,1),(0,0),3(28,8),3(24,7)\},
\end{align*}

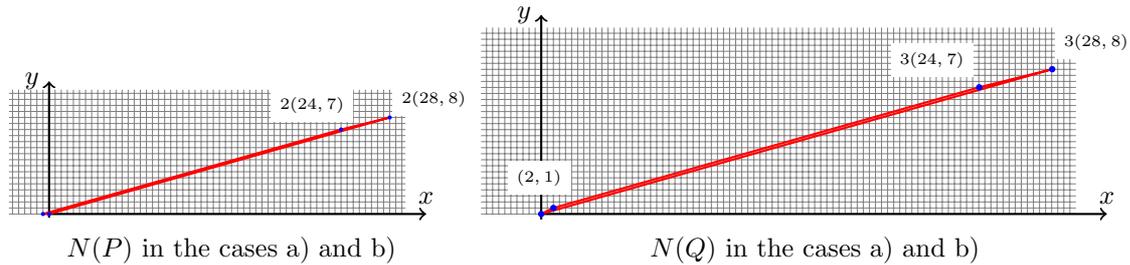
\noindent \begin{tikzpicture}[scale=0.16]
\draw(34,0)node{};
\draw[step=0.5cm,gray,very thin] (-3.3,-0.03) grid (29.3,10.3);
\draw [thick,->] (0,0)--(31,0) node[anchor=south]{$x$};
\draw [thick,->] (0,-0.3) --(0,11) node[anchor=east]{$y$};
\draw [red,thick,-] (-0.5,0)--(24,7)--(28,8)--(0,0)--(-0.5,0);
\fill[blue] (0,0) circle (5pt)
(-0.5,0) circle (5pt)
(24,7) circle (5pt)
(28,8) circle (5pt);
\draw (28,9.5) node[fill=white,right=1pt] {\tiny $2(28,8)$};
\draw (18,9) node[fill=white,right=1pt] {\tiny $2(24,7)$};
\draw (15,-3)node{$N(P)$ in the cases a) and b)};
\end{tikzpicture}
\begin{tikzpicture}[scale=0.24]
\draw(34,0)node{};
\draw[step=0.333cm,gray,very thin] (-3.3,-0.03) grid (29.3,10.3);
\draw [thick,->] (0,0)--(31,0) node[anchor=south]{$x$};
\draw [thick,->] (0,-0.3) --(0,11) node[anchor=east]{$y$};
\draw [red,thick,-] (0,0)--(0.667,0.333)--(24,7)--(28,8)--(0,0);
\fill[blue] (0,0) circle (5pt)
(0.667,0.333) circle (5pt)
(24,7) circle (5pt)
(28,8) circle (5pt);
\draw (-2,2) node[fill=white,right=1pt] {\tiny $(2,1)$};
\draw (28,9.5) node[fill=white,right=1pt] {\tiny $3(28,8)$};
\draw (19,8.5) node[fill=white,right=1pt] {\tiny $3(24,7)$};
\draw (15,-2)node{$N(Q)$ in the cases a) and b)};
\end{tikzpicture}

 whereas in the third
case c) we obtain
\begin{align*}
   & N(P)=\{(-1,0),(0,0),2(16,4),2(28,8),2(24,7)\}\\
   & N(Q)=\{(2,1),(0,0),3(16,4),3(28,8),3(24,7)\}.
\end{align*}

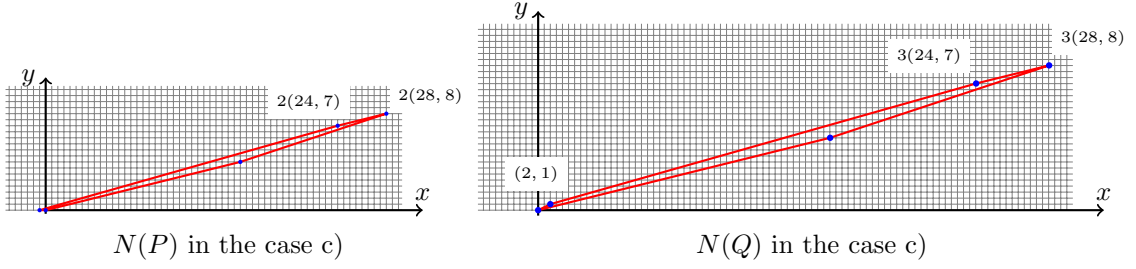
\noindent \begin{tikzpicture}[scale=0.16]
\draw(34,0)node{};
\draw[step=0.5cm,gray,very thin] (-3.3,-0.03) grid (29.3,10.3);
\draw [thick,->] (0,0)--(31,0) node[anchor=south]{$x$};
\draw [thick,->] (0,-0.3) --(0,11) node[anchor=east]{$y$};
\draw [red,thick,-] (-0.5,0)--(24,7)--(28,8)--(16,4)--(0,0)--(-0.5,0);
\fill[blue] (0,0) circle (5pt)
(-0.5,0) circle (5pt)
(16,4) circle (5pt)
(24,7) circle (5pt)
(28,8) circle (5pt);
\draw (28,9.5) node[fill=white,right=1pt] {\tiny $2(28,8)$};
\draw (18,9) node[fill=white,right=1pt] {\tiny $2(24,7)$};
\draw (15,-3)node{$N(P)$ in the case c)};
\end{tikzpicture}
\begin{tikzpicture}[scale=0.24]
\draw(34,0)node{};
\draw[step=0.333cm,gray,very thin] (-3.3,-0.03) grid (29.3,10.3);
\draw [thick,->] (0,0)--(31,0) node[anchor=south]{$x$};
\draw [thick,->] (0,-0.3) --(0,11) node[anchor=east]{$y$};
\draw [red,thick,-] (0,0)--(0.667,0.333)--(24,7)--(28,8)--(16,4)--(0,0);
\fill[blue] (0,0) circle (5pt)
(0.667,0.333) circle (5pt)
(16,4) circle (5pt)
(24,7) circle (5pt)
(28,8) circle (5pt);
\draw (-2,2) node[fill=white,right=1pt] {\tiny $(2,1)$};
\draw (28,9.5) node[fill=white,right=1pt] {\tiny $3(28,8)$};
\draw (19,8.5) node[fill=white,right=1pt] {\tiny $3(24,7)$};
\draw (15,-2)node{$N(Q)$ in the case c)};
\end{tikzpicture}

Apply the the morphism $\varphi$ with $\varphi(x) = x^{-1}$ and $\varphi(y) = x^4 y$. As in Proposition ~\ref{case927}, this is not an automorphism and the chain rule gives $[\varphi(P), \varphi(Q)] = -[P,Q] x^2$. A straightforward computation shows that in the cases a) and b)
 the Newton Polygons of $P$ and $Q$ become
\begin{align*}
	& N(P) = \{(0,0),(1,0),(8,14), (8,16)\} \\
	& N(Q) = \{(0,0),(2,1),(12,21), (12,24)\},
\end{align*}

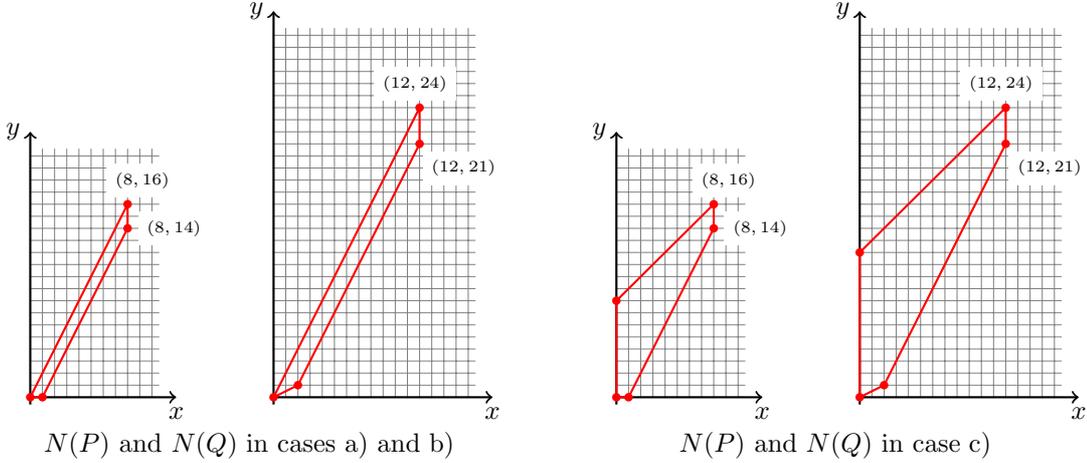
\noindent \begin{tikzpicture}[scale=0.32]
\draw[step=0.5cm,gray,very thin] (-0.03,-0.03) grid (5.3,10.3);
\draw [thick,->] (0,0)--(6,0) node[anchor=north]{$x$};
\draw [thick,->] (0,-0.3) --(0,11) node[anchor=east]{$y$};
\draw [red,thick,-] (0,0) -- (0.5,0)--(4,7)--(4,8)--(0,0);
\fill[red] (0,0) circle (5pt)
(0.5,0) circle (5pt)
(4,7) circle (5pt)
(4,8) circle (5pt);
\draw (3,9) node[fill=white,right=1pt] {\tiny $(8,16)$};
\draw (4.3,7) node[fill=white,right=1pt] {\tiny $(8,14)$};
\draw[step=0.5cm,gray,very thin] (10-0.03,-0.03) grid (18.3,15.3);
\draw [thick,->] (10,0)--(19,0) node[anchor=north]{$x$};
\draw [thick,->] (10,-0.3) --(10,16) node[anchor=east]{$y$};
\draw [red,thick,-] (10,0) -- (11,0.5)--(16,10.5)--(16,12)--(10,0);
\fill[red] (10,0) circle (5pt)
(11,0.5) circle (5pt)
(16,10.5) circle (5pt)
(16,12) circle (5pt);
\draw (14,13) node[fill=white,right=1pt] {\tiny $(12,24)$};
\draw (16,9.5) node[fill=white,right=1pt] {\tiny $(12,21)$};
\draw (9,-2) node{$N(P)$ and $N(Q)$ in cases a) and b)};
\draw (22,9.5) node {};
\end{tikzpicture}
\begin{tikzpicture}[scale=0.32]
\draw[step=0.5cm,gray,very thin] (-0.03,-0.03) grid (5.3,10.3);
\draw [thick,->] (0,0)--(6,0) node[anchor=north]{$x$};
\draw [thick,->] (0,-0.3) --(0,11) node[anchor=east]{$y$};
\draw [red,thick,-] (0,0) -- (0.5,0)--(4,7)--(4,8)--(0,4)--(0,0);
\fill[red] (0,0) circle (5pt)
(0.5,0) circle (5pt)
(4,7) circle (5pt)
(0,4) circle (5pt)
(4,8) circle (5pt);
\draw (3,9) node[fill=white,right=1pt] {\tiny $(8,16)$};
\draw (4.3,7) node[fill=white,right=1pt] {\tiny $(8,14)$};
\draw[step=0.5cm,gray,very thin] (10-0.03,-0.03) grid (18.3,15.3);
\draw [thick,->] (10,0)--(19,0) node[anchor=north]{$x$};
\draw [thick,->] (10,-0.3) --(10,16) node[anchor=east]{$y$};
\draw [red,thick,-] (10,0) -- (11,0.5)--(16,10.5)--(16,12)--(10,6)--(10,0);
\fill[red] (10,0) circle (5pt)
(11,0.5) circle (5pt)
(16,10.5) circle (5pt)
(10,6) circle (5pt)
(16,12) circle (5pt);
\draw (14,13) node[fill=white,right=1pt] {\tiny $(12,24)$};
\draw (16,9.5) node[fill=white,right=1pt] {\tiny $(12,21)$};
\draw (9,-2) node{$N(P)$ and $N(Q)$ in case c)};

\end{tikzpicture}

and in the case c) it becomes
\begin{align*}
&	N(P) = \{(0,0),(1,0),(8,14), (8,16), (0,8)\} \\
&	N(Q) = \{(0,0),(2,1),(12,21), (12,24), (0,12)\}.
\end{align*}
\end{proof}

\begin{proposition}[Case (7,21)]
	If there is a counterexample to the Jacobian Conjecture in the case $(7,21)$, then there exist $P,Q \in K[x,y]$ with $[P,Q] = x$ and
\begin{align*}
	N(P) = \{(0,0),2(2,0),2(3,1), 2(0,7)\} \\
	N(Q) = \{(0,0),3(2,0),3(3,1), 3(0,7)\}
\end{align*}	
\end{proposition}
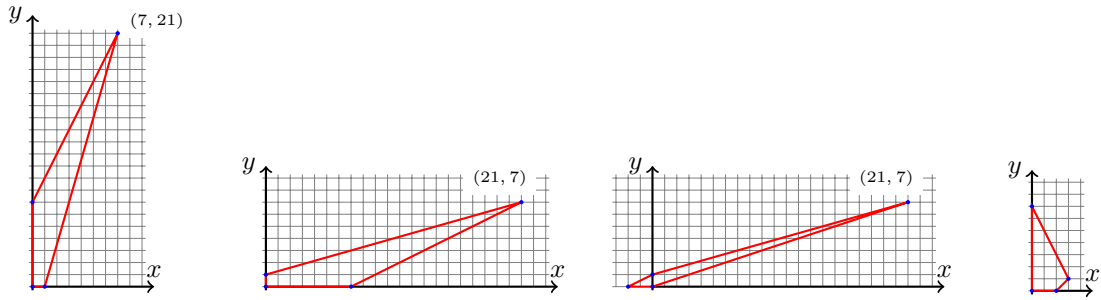
\begin{figure}[htb]
\centering
 \begin{tikzpicture}[scale=0.16]
\draw(15,0)node{};
\draw[step=1cm,gray,very thin] (-0.3,-0.03) grid (9.3,21.3);
\draw [thick,->] (0,0)--(10,0) node[anchor=south]{$x$};
\draw [thick,->] (0,-0.3) --(0,22.5) node[anchor=east]{$y$};
\draw [red,thick,-] (0,0)--(0,7)--(7,21)--(1,0)--(0,0);
\fill[blue] (0,0) circle (5pt)
(0,7) circle (5pt)
(7,21) circle (5pt)
(1,0) circle (5pt);
\draw (7,22) node[fill=white,right=1pt] {\tiny $(7,21)$};
\end{tikzpicture}
 \begin{tikzpicture}[scale=0.16]
\draw(27,0)node{};
\draw[step=1cm,gray,very thin] (-0.3,-0.03) grid (23.3,9.3);
\draw [thick,->] (0,0)--(24,0) node[anchor=south]{$x$};
\draw [thick,->] (0,-0.3) --(0,10) node[anchor=east]{$y$};
\draw [red,thick,-] (0,0)--(7,0)--(21,7)--(0,1)--(0,0);
\fill[blue] (0,0) circle (5pt)
(7,0) circle (5pt)
(21,7) circle (5pt)
(0,1) circle (5pt);
\draw (16,9) node[fill=white,right=1pt] {\tiny $(21,7)$};
\end{tikzpicture}
\begin{tikzpicture}[scale=0.16]
\draw(27,0)node{};
\draw[step=1cm,gray,very thin] (-3.3,-0.03) grid (23.3,9.3);
\draw [thick,->] (0,0)--(24,0) node[anchor=south]{$x$};
\draw [thick,->] (0,-0.3) --(0,10) node[anchor=east]{$y$};
\draw [red,thick,-] (-2,0)--(0,0)--(21,7)--(0,1)--(-2,0);
\fill[blue] (0,0) circle (5pt)
(-2,0) circle (5pt)
(21,7) circle (5pt)
(0,1) circle (5pt);
\draw (16,9) node[fill=white,right=1pt] {\tiny $(21,7)$};
\end{tikzpicture}
\begin{tikzpicture}[scale=0.16]
\draw[step=1cm,gray,very thin] (-0.3,-0.03) grid (4.3,9.3);
\draw [thick,->] (0,0)--(5,0) node[anchor=south]{$x$};
\draw [thick,->] (0,-0.3) --(0,10) node[anchor=east]{$y$};
\draw [red,thick,-] (0,0)--(2,0)--(3,1)--(0,7)--(0,0);
\fill[blue] (0,0) circle (5pt)
(2,0) circle (5pt)
(0,7) circle (5pt)
(3,1) circle (5pt);
\end{tikzpicture}
\caption{The transformation of $\frac 12 N(P)=\frac 13 N(Q)$}
\end{figure}

\begin{proof}
The corners of the polygons of $P$ and $Q$ in this case are $\{ (0,0), (1,0), (7,21), (0,7) \}$ multiplied by $(m,n) = (2,3)$ respectively (This time, $\en_{7,-2} F = \frac{5}{7}(7,21)$ for the relevant direction, and so $q=7$).
After the automorphism $\phi_1$ with $\phi_1(x) = y$ and $\phi_1(y) = x$, we have the polygon $\{ (0,0), (7,0), (21,7), (0,1) \}$. \\
Since $q=7$ for the direction $(\rho,\sigma)=(-2,7)$, by~\cite{GGV1}*{Corollary 7.4} we know that
$$
\ell_{1,-2}(P)=x^{28}\lambda_p (y-\lambda x^{-2})^{14}\quad\text{and}\quad\ell_{1,-2}(Q)=x^{42}\lambda_q (y-\lambda x^{-2})^{21}.
$$
Now we apply the automorphism $\phi_2$ of $L^{(1)}$ given by $\phi_2(x) = x$, $\phi_2(y) = y +\lambda x^{-2}$ and obtain the polygon $\{ (-2,0), (0,0), (21,7), (0,1) \}$
Finally apply the automorphism $\phi_3$ of $L^{(1)}$ given by $\phi_3(x) = x^{-1}$, $\phi_3(y) = y x^3$ and obtain the polygon $\{ (0,0), (2,0), (3,1), (0,7) \}$, as desired.
\end{proof}
\newpage

\section{Systems of polynomial equations for $(9,24)$ and $(9,27)$}
\label{seccion 4}
In this section we will prove the following theorem.
\begin{theorem}\label{teorema impossible}
There exist no pair of polynomials $P,Q\in K[x,y]$ such that
\begin{enumerate}
  \item $[P,Q]=x+g(y)$ for some $g(y)\in K[y]$,
  \item $\en_{3,-1}(P)=\st_{1,0}(P)= 2(3,8)=(6,16)$ and $\st_{-1,1}(P)=\en_{1,0}(P)= 2(3,9)=(6,18)$,
  \item $\en_{3,-1}(Q)=\st_{1,0}(Q)= 3(3,8)=(9,24)$ and $\st_{-1,1}(Q)=\en_{1,0}(Q)= 3(3,9)=(9,27)$,
\end{enumerate}
\end{theorem}

So we assume that such a pair exists, and we will  arrive at a contradiction.
By~\cite{GGV1}*{Propositions 1.13 and 2.1} there exists a $(1,0)$-homogeneous $R$ with $\ell_{1,0}(P)= R^2$ and $\ell_{1,0}(Q)=R^3$. Necessarily $R=x^3 C_3$ with $C_3= y^8(a_0+a_1 y)$
for some $a_0,a_1\in K^{\times}$,
and by a linear change of variables we may also assume that $C_3= y^8 (y+1)$.

\begin{proposition}\label{calculo de C}
 There exist $C,F\in K[y,C_3^{-1}]((x^{-1}))$
with $v_{1,0}(F)=-4$ and $\alpha_2,\alpha_1,\alpha_0,\alpha_{-1}\in K$ such that
  \begin{enumerate}
    \item $P=C^2$ and $Q=C^3 +\alpha_2 C^2+\alpha_1 C+\alpha_0+\alpha_{-1}C^{-1}+F$,
    \item $\ell_{1,0}(C)=x^3 C_3$,
    \item $v_{-1,1}(C)=v_{-1,1}(3,9)=6$, if we consider $K[y,C_3^{-1}]\subset K((y^{-1}))$,
  \item $v_{3,-1}(C)=v_{3,-1}(3,8)=1$, if we consider $K[y,C_3^{-1}]\subset K((y))$.
  \end{enumerate}
\end{proposition}

\begin{proof}
We will construct inductively
$$
C=x^3 C_3+x^2 C_2+x C_1 +C_0+ x^{-1}C_{-1}+\dots,
$$
with $C_k\in K[y,C_3^{-1}]$ such that $C^2=P$ and such that items~(3) and~(4) are satisfied. It suffices to find $C_{3-k}$ for $k\ge 0$ such that
\begin{enumerate}
  \item[$a_k$)] $\left(\left(\sum_{i=0}^{k} C_{3-i}x^{3-i}\right)^2\right)_{6-j}=P_{6-j}$ for $j=0,\dots k$,
  \item[$b_k$)] $v_{-1,1}(x^{3-j} C_{3-j})\le 6$ for $j=0,\dots, k$, if we consider $K[y,C_3^{-1}]\subset K((y^{-1}))$, and
  \item[$c_k$)] $v_{3,-1}(x^{3-j} C_{3-j})\le 1$ for $j=0,\dots, k$, if we consider $K[y,C_3^{-1}]\subset K((y))$.
\end{enumerate}

We already have
\begin{enumerate}
  \item[$a_0$)] $(x^3 C_3)^2=x^6 P_6$,
  \item[$b_0$)] $v_{-1,1}(x^3 C_3)=v_{-1,1}(3,9)=6$ and
  \item[$c_0$)] $v_{3,-1}(x^3 C_3)=v_{3,-1}(3,8)=1$.
\end{enumerate}
 Assume that we have constructed $C_3,\dots,C_{3-k+1}\in K[y,C_3^{-1}]$ such that
\begin{enumerate}
  \item[$a_{k-1}$)] $\left(\left(\sum_{i=0}^{k-1} C_{3-i}x^{3-i}\right)^2\right)_{6-j}=P_{6-j}$ for $j=0,\dots k-1$,
  \item[$b_{k-1}$)] $v_{-1,1}(x^{3-j} C_{3-j})\le 6$  for $j=0,\dots, k-1$, if we consider $K[y,C_3^{-1}]\subset K((y^{-1}))$, and
  \item[$c_{k-1}$)] $v_{3,-1}(x^{3-j} C_{3-j})\le 1$  for $j=0,\dots, k-1$, if we consider $K[y,C_3^{-1}]\subset K((y))$.
\end{enumerate}
Then we want to find $C_{3-k}$ such that
$$
\left(\left(\sum_{i=0}^{k} C_{3-i}x^{3-i}\right)^2\right)_{6-j}=P_{6-j}\quad\text{for}\quad j=0,\dots, k.
$$
But the term $2C_{3-k}x^{3-k} C_3x^3=2C_{3-k}C_3x^{6-k}$ contains the highest power of $x$ whose coefficient
is touched by $C_{3-k}$. Hence it suffices to guarantee
$$
P_{6-k}=\left(\left(\sum_{i=0}^{k} C_{3-i}x^{3-i}\right)^2\right)_{6-k}=\sum_{j=0}^{k}C_{3-j}C_{3-k+j}.
$$
So we want
$$
2C_3 C_{3-k}=-P_{6-k}-\sum_{j=1}^{k-1}C_{3-j}C_{3-k+j}
$$
and it suffices to set
\begin{equation}\label{ecuacion de C 3-k}
C_{3-k}:=-\frac{1}{2C_3}\left( P_{6-k}+\sum_{j=1}^{k-1}C_{3-j}C_{3-k+j}\right)\in K[y,C_3^{-1}],
\end{equation}
and we obtain that $a_k$) is satisfied.

If we consider $K[y,C_3^{-1}]\subset K((y^{-1}))$, we have $v_{-1,1}(C_3^{-1})=-9$, since
$$
C_3^{-1}=y^{-9}(1+y^{-1})^{-1}=y^{-9}(1-y^{-1}+y^{-2}-y^{-3}+y^{-4}-\dots)
$$
and by induction hypothesis we also have $v_{-1,1}(C_{3-j}x^{3-j})\le 6$, hence $v_{-1,1}(C_{3-j})+j-3\le 6$, and so $v_{-1,1}(C_{3-j})\le 9-j$ for $j=1,\dots ,k-1$.
This implies
$$
v_{-1,1}(C_{3-j}C_{3-k+j})\le  v_{-1,1}(C_{3-j}) v_{-1,1}(C_{3-k+j})\le 9-j+9-k+j=18-k
$$
We also have $v_{-1,1}(P_{6-k}x^{6-k})\le 12$, hence $v_{-1,1}(P_{6-k})+k-6\le 12$, and so $v_{-1,1}(P_{6-k})\le 18-k$.
Now from ~\eqref{ecuacion de C 3-k} it follows that
$$
v_{-1,1}(C_{3-k})\le -v_{-1,1}(C_3)+v_{-1,1}\left( P_{6-k}+\sum_{j=1}^{k-1}C_{3-j}C_{3-k+j}\right)\le -9+18-k=9-k,
$$
which implies $v_{-1,1}(C_{3-k}x^{3-k})\le 6$ and so~$b_k$) holds.

Similarly, if we consider $K[y,C_3^{-1}]\subset K((y))$, we have $v_{3,-1}(C_3^{-1})=8$, since
$$
C_3^{-1}=y^{-8}(1+y)^{-1}=y^{-8}(1-y+y^{2}-y^{3}+y^{4}-\dots)
$$
and by induction hypothesis we also have $v_{3,-1}(C_{3-j}x^{3-j})\le 1$, hence $v_{3,-1}(C_{3-j})+9-3j\le 1$, and so $v_{3,-1}(C_{3-j})\le -8+3j$ for $j=1,\dots ,k-1$.
This implies
$$
v_{3,-1}(C_{3-j}C_{3-k+j})\le  v_{3,-1}(C_{3-j}) v_{3,-1}(C_{3-k+j})\le -8+3j-8+3(k-j)=-16+3k
$$
We also have $v_{3,-1}(P_{6-k}x^{6-k})\le 2$, hence $v_{3,-1}(P_{6-k})+18-3k\le 2$, and so $v_{3,-1}(P_{6-k})\le -16+3k$.
Now from ~\eqref{ecuacion de C 3-k} it follows that
$$
v_{3,-1}(C_{3-k})\le v_{3,-1}(C_3^{-1})+v_{3,-1}\left( P_{6-k}+\sum_{j=1}^{k-1}C_{3-j}C_{3-k+j}\right)\le 8-16+3k=-8+3k,
$$
which implies $v_{3,-1}(C_{3-k}x^{3-k})\le 1$ and so~$c_k$) holds.

In order to finish the proof it suffices to find $F\in K[y,C_3^{-1}]((x^{-1}))$
with $v_{1,0}(F)=-4$ such that
\begin{equation}\label{igualdad para Q}
Q=C^3 +\alpha_2 C^2+\alpha_1 C+\alpha_0+\alpha_{-1}C^{-1}+F.
\end{equation}

Since $\ell_{1,0}(Q)=\ell_{1,0}(C^3)$ we have that $v_{1,0}(Q-C^3)<9$. If $v_{1,0}(Q-C^3)>-4$, then
$$
v_{1,0}(Q-C^3)+v_{1,0}(P)-v_{1,0}(1,1)>-4+6-1=1=v_{1,0}(x)= v_{1,0}([Q-C^3,P])
$$
which by~\cite{GGV1}*{Proposition 1.13} implies $[\ell_{1,0}(P), \ell_{1,0}(Q-C^3)]=0$, and so, by~\cite{GGV1}*{Proposition 2.1} and the fact that $x^3C_3$ is not the positive
power of any element of $K[x,y]$, we have
$$
\ell_{1,0}(Q-C^3)=\alpha_k (x^3 C_3)^k,
$$
for some $k$ with $-2<k<3$. Using the arguments of \cite{GGV3}*{Section 1}, we find $\alpha_{2},\alpha_{1},\alpha_{0},\alpha_{-1}\in K$ such that
$$
[\ell_{1,0}(P), \ell_{1,0}(Q-C^3-\alpha_2 C^2-\alpha_1 C-\alpha_0-\alpha_{-1}C^{-1})]\ne 0,
$$
and so, again by~\cite{GGV1}*{Proposition 1.13} we know that  $F:=Q-C^3-\alpha_2 C^2-\alpha_1 C-\alpha_0-\alpha_{-1}C^{-1}$ satisfies
$$
v_{1,0}(F)+v_{1,0}(P)-v_{1,0}(1,1)=v_{1,0}([F,P]),
$$
which implies $v_{1,0}(F)=-4$ and concludes the proof.
\end{proof}

\begin{remark}
  Note that if we replace $Q$ by $\tilde Q := Q - \alpha_2 P - \alpha_0$ and $P$ by $\tilde P:= P+\frac 23 \alpha_1$, then
$$
\tilde Q = C^3 +  \alpha_1 C +  \lambda C^{-1} + F.
$$
Moreover $\tilde C$ with $\tilde C^2=\tilde P$ is given by a series
$$
\tilde C=C+\frac 13 \alpha_1 C^{-1}+\gamma_3 C^{-3}+\gamma_5 C^{-5}+\dots
$$
and so
$$
\tilde C^3=C^3+\alpha_1 C+\tilde\gamma_1 C^{-1}+\tilde\gamma_3 C^{-3}+\dots.
$$
It follows that
$$
\tilde Q=\tilde C^3+\lambda \tilde C^{-1}+\tilde F,\quad\text{for some $\lambda\in K$}.
$$
Since  $\tilde P$, $\tilde Q$ satisfy the conditions of Theorem~\ref{teorema impossible}, we can and will assume that $\alpha_2,\alpha_1,\alpha_0$ vanish in Proposition~\ref{calculo de C}.
\end{remark}

By definition $F_{-4}\in K[y,C_3^{-1}]$.

\begin{proposition}
Set $f:=C_3^2 F_{-4}$ then $f\in K[y]$ is a separable polynomial (it has no multiple roots), $y(y+1)$ divides $f$ and $\deg(f)=6$.
\end{proposition}

\begin{proof}
Note that $Q^2=C^6+2\lambda C^2+2C^3 F+\lambda^2 C^{-2}+2\lambda C^{-1}F+F^2$, and so
\begin{equation}\label{F casi esta en K[y]}
  \ell_{1,0}(Q^2-P^3-2\lambda P)=\ell_{1,0}(2C^3 F)=2x^5 C_3^3 F_{-4}\in K[x,y]
\end{equation}
Moreover, $[\ell_{1,0}(P),\ell_{1,0}(F)]\ne 0$, and so
$$
[\ell_{1,0}(P),\ell_{1,0}(Q^2-P^3-2\lambda P)]\ne 0.
$$
This implies that
$$
\ell_{1,0}[P,Q^2-P^3-2\lambda P]=[\ell_{1,0}(P),\ell_{1,0}(Q^2-P^3-2\lambda P)]=[x^6 C_3^2,2x^5 C_3^3 F_{-4}]
$$
and so
\begin{equation}\label{F con cunmutador}
  2x^{10} C_3^3=\ell_{1,0}(2x Q)=\ell_{1,0}([P,Q^2])=\ell_{1,0}([P,Q^2-P^3-2\lambda P)])=[x^6 C_3^2,2x^5 C_3^3 F_{-4}]
\end{equation}
Set $f_1:= C_3^3 F_{-4}$. By~\eqref{F casi esta en K[y]} we know that $f_1\in K[y]$. From~\eqref{F con cunmutador} we obtain
$$
x^{10} C_3^3=2C_3 x^3[x^3 C_3,x^5 f_1],
$$
and so
\begin{eqnarray*}
x^7 C_3^2 &=& 2 [x^3 C_3,x^5 f_1] \\
   &=& 2[x^3,f_1]C_3 x^5+2[C_3,x^5]x^3 f_1 \\
   &=& 6x^2 [x,f_1]C_3 x^5+10x^4[C_3,x]x^3 f_1 \\
   &=& x^7 (6C_3 f_1'-10 C_3' f_1).
\end{eqnarray*}
Since $C_3'=(9y+8)y^7$, this yields
$$
y^{16}(y+1)^2=6y^8(y+1)f_1'-10(9y+8)y^7 f_1
$$
and so we arrive at the following differential equation for $f_1$:
$$
y^9(y+1)^2=6y(y+1)f_1'-10(9y+8) f_1.
$$
This equation has a unique solution that can be found using a CAS:
$$
f_1=-\frac{1}{910}y^9(y+1)^2(35-42y+54 y^2-81 y^3+243 y^4),
$$
and one also checks that neither $y$ nor $(y+1)$ divide $g(y):=-\frac{1}{910}(35-42y+54 y^2-81 y^3+243 y^4)$, which
has four different (complex) roots.
So $f:=\frac{f_1}{C_3}$ has the desired properties.
\end{proof}

The equalities
\begin{equation}\label{ecuaciones basicas}
C^2=P\quad\text{and}\quad Q=C^3+\lambda C^{-1}+F
\end{equation}
 yield a system of polynomial equations for $C_k$, similar to the systems of~\cite{GGV3}, which correspond to
the equalities
\begin{equation}\label{ecuaciones polinomiales}
P_{-k}=0\quad\text{for}\quad k=1,\dots,8\quad\text{and}\quad Q_{-k}=0\quad\text{for}\quad k=1,\dots 5.
\end{equation}

However $C_k$ in general is not a polynomial, so we will transform the system into a system for certain $D_k$'s which are polynomials.

\begin{proposition}
Set $D_k:=C_k C_3^{5-2k}$. Then $D_k\in K[y]$.
\end{proposition}

\begin{proof}
  By definition $D_3=1$ and $D_2=C_2 C_3=\frac 12 P_5$, so $D_1, D_2\in K[y]$.
  Now assume that $D_j\in K[y]$ for $j>k$. Then
  $$
  P_{k+3}=\sum_{j=k}^{3} C_j C_{k+3-j}=2C_3 C_k+\sum_{j=k+1}^{2} C_j C_{k+3-j}.
  $$
  Note that $P_k\in K[y]$ for all $k$. In fact, for $k<0$, $P_k=0\in K[y]$.
  Then
  $$
  D_k=C_3^{5-2k}C_k=\frac 12\left(P_{k+3}-\sum_{j=k+1}^{2} C_j C_{k+3-j} \right)C_3^{4-2k}=\frac 12 P_{k+3}C_3^{4-2k}-\sum_{j=k+1}^{2} D_j D_{k+3-j},
  $$
  since
  $$
  D_j D_{k+3-j}=C_j C_{k+3-j}C_3^{5-2j}C_3^{5-2(k+3-j)}=C_j C_{k+3-j}C_3^{4-2k}.
  $$
  This shows inductively that $D_k\in K[y]$.
\end{proof}

If we set $D:=\sum_{k\le 3} x^k D_k$, then
$$
(D^2)_{-k}=\sum_{\overset{i,j\le 3}{i+j=-k}}D_i D_j=\sum_{\overset{i,j\le 3}{i+j=-k}}C_i C_3^{5-2i}C_jC_3^{5-2j}=\sum_{\overset{i,j\le 3}{i+j=-k}}C_i C_jC_3^{10-2(i+j)}=(C^2)_{-k}C_3^{10+2k}
$$
and
$$
(D^3)_{-k}=\sum_{\overset{i,j,l\le 3}{i+j+l=-k}}D_i D_jD_l=\sum_{\overset{i,j,l\le 3}{i+j+l=-k}}C_i C_jC_lC_3^{15-2(i+j+l)}=(C^3)_{-k}C_3^{15+2k}.
$$
Now the equalities~\eqref{ecuaciones polinomiales} imply that the following equalities
\begin{equation}\label{ecuaciones polinomiales nuevas}
(D^2)_{-k}=0\quad\text{for}\quad k=1,\dots,8,\quad (D^3)_{-1}=0,\quad (D^3)_{-2}=0,
\end{equation}
$$
(D^3)_{-3}+\lambda C_3^{20}=0\quad\text{and}\quad  (D^3)_{-4}-\lambda D_2 C_3^{20}+F_{-4}C_3^{23}=0
$$
hold. Note that we don't use the equality $Q_{-5}=0$.

So there exists $D,\hat F\in K[y]((x^{-1}))$, such that $\ord_{x^{-1}}(\hat F)\ge 5$, $\hat P:= D^2\in K[x,y]$ and
$$
\hat Q:=D^3+D^{-1}\lambda C_3^{20}+x^{-4}F_{-4}C_3^{23}+\hat F\in K[x,y].
$$
In fact, $P=C^2\in K[x,y]$ implies that $(C^2)_{-k}=0$ for all $k>0$, hence $(D^2)_{-k}=0$ for all $k>0$, so $\hat P\in K[x,y]$;
on the other hand the equalities~\eqref{ecuaciones polinomiales nuevas} imply that
$(\hat Q)_{-k}=0$ for $k=1,2,3,4$. Now set
$$
(\hat F)_{-k}:=-(D^3+D^{-1}\lambda C_3^{20}+F_{-4}C_3^{23})_{-k}\quad\text{for $k\ge 5$}
$$
which implies that $\hat Q\in K[x,y]$, as desired.

Consider the automorphism $\varphi$ of $K[y]((x^{-1}))$ given by $\varphi(y)=y$, $\varphi(x)=x-D_2$ (and so $\varphi(x^{-1})=x^{-1}+x^{-2}D_2+x^{-3}D_2^2+x^{-4}D_2^3+\dots$).
Then $\varphi(\hat P),\varphi(\hat Q)\in K[x,y]$, and if we set $\tilde D:=\varphi(D)$, then
$$
\tilde D=x^3+x d_1 +d_0+x^{-1}d_{-1}+x^{-2}d_{-2}+\dots
$$
for some $d_i\in K[y]$ (note that $d_2=0$), and the conditions $\varphi(\hat P),\varphi(\hat Q)\in K[x,y]$ imply the equalities
$$
(\tilde D^2)_{-k}=0\quad\text{for}\quad k=1,\dots,8\quad\text{and}\quad (\tilde D^3+\tilde D^{-1}\lambda C_3^{20}+x^{-4}F_{-4}C_3^{23})_{-j}=0\quad\text{for}\quad j=1,\dots 4.
$$
Note that we can take $x^{-4}$ instead of $\varphi(x^{-4})$ and that we can ignore the term $\varphi(\hat F)$, since we only consider $j\le 4$.
So we obtain the following equalities:
\begin{align*}
  0=(\tilde D^2)_{-1} & = 2 d_0 d_{-1} + 2 d_1 d_{-2} + 2 d_{-4}\\
0=(\tilde D^2)_{-2} &= d_{-1}^2 + 2 d_0 d_{-2} + 2 d_1 d_{-3} + 2 d_{-5}\\
0=(\tilde D^2)_{-3} &= 2 d_{-1} d_{-2} + 2 d_0 d_{-3} + 2 d_1 d_{-4} + 2 d_{-6}\\
0=(\tilde D^2)_{-4} &= d_{-2}^2 + 2 d_{-1} d_{-3} + 2 d_0 d_{-4} + 2 d_1 d_{-5} + 2 d_{-7}\\
0=(\tilde D^2)_{-5} &= 2 d_{-2} d_{-3} + 2 d_{-1} d_{-4} + 2 d_0 d_{-5} + 2 d_1 d_{-6} + 2 d_{-8}\\
0=(\tilde D^2)_{-7} &= 2 d_{-1}0 + 2 d_{-3} d_{-4} + 2 d_{-2} d_{-5} + 2 d_{-1} d_{-6} + 2 d_0 d_{-7} +  2 d_1 d_{-8}\\
0=(\tilde Q)_{-1}&= 3 d_0^2 d_{-1} + 3 d_1 d_{-1}^2 + 6 d_0 d_1 d_{-2} + 3 d_{-2}^2 + 3 d_1^2 d_{-3} + 6 d_{-1} d_{-3} + 6 d_0 d_{-4} + 6 d_1 d_{-5}\\
&\quad + 3 d_{-7}\\
0=(\tilde Q)_{-2}&= 3 d_0 d_{-1}^2 + 3 d_0^2 d_{-2} + 6 d_1 d_{-1} d_{-2} + 6 d_0 d_1 d_{-3} +  6 d_{-2} d_{-3} + 3 d_1^2 d_{-4} + 6 d_{-1} d_{-4}\\
&\quad + 6 d_0 d_{-5} + 6 d_1 d_{-6} +  3 d_{-8}\\
0=(\tilde Q)_{-4}&= 3 d_{-10} + 3 d_{-1}^2 d_{-2} + 3 d_0 d_{-2}^2 + 6 d_0 d_{-1} d_{-3} + 6 d_1 d_{-2} d_{-3} +  3 d_0^2 d_{-4} + 6 d_1 d_{-1} d_{-4}\\
&\quad + 6 d_{-3} d_{-4} + 6 d_0 d_1 d_{-5} + 6 d_{-2} d_{-5} +  3 d_1^2 d_{-6} + 6 d_{-1} d_{-6} + 6 d_0 d_{-7} + 6 d_1 d_{-8}\\
&\quad + F_{-4}C_3^{23}.
\end{align*}
We consider this as a system of 9 equations and using a CAS (for example Mathematica) we eliminate the variables $d_{-10}, d_{-8}, d_{-7}, d_{-6}, d_{-5}, d_{-4}, d_{-3}, d_{-2}$, obtaining

\begin{equation}\label{ecuacion principal}
  18 C_3^{23} d_1 (d_{-1})^6 F_{-4} + 8 C_3^{69} F_{-4}^3 + 27 d_0 (d_{-1})^9=0.
\end{equation}

\begin{proposition}
  We have $v_{-13,-1}(D)=-39$ and $v_{17,1}(D)=51$.
\end{proposition}

\begin{proof}
On the one hand, if we embed $C_3^{-1}$ in $K((y))$, then by Proposition~\ref{calculo de C} we know that $v_{3,-1}(C)=1$.  Thus the lowest power of $y$ of a term of
$C_k$ is at least $y^{3k-1}$, since
$$
v_{3,-1}(x^k y^{3k-1})=3k-(3k-1)=1=v_{3,-1}(C).
$$
But then
\begin{eqnarray*}
v_{-13,-1}(D_k x^k) &=& v_{-13,-1}(C_k)+v_{-13,-1}(x^k)+v_{-13,-1}(C_3^{5-2k}) \\
   &\le & -(3k-1)-13k+(5-2k)v_{-13,-1}(C_3)\\
   &=& -16k+1+(5-2k)(-8)\\
   &=& -39.
\end{eqnarray*}
Since $v_{-13,-1}(D_3)=-39$, this proves $v_{-13,-1}(D)=-39$.

On the other hand, if we embed $C_3^{-1}$ in $K((y^{-1}))$, then by Proposition~\ref{calculo de C} we know that $v_{-1,1}(C)=6$. Thus the highest power of $y$ of a term of
$C_k$ is at most $y^{k+6}$, since
$$
v_{-1,1}(x^k y^{k+6})=-k+k+6=6=v_{-1,1}(C).
$$
But then
\begin{eqnarray*}
v_{17,1}(D_k x^k) &=& v_{17,1}(C_k)+v_{17,1}(x^k)+v_{17,1}(C_3^{5-2k}) \\
   &\le & k+6 +17k+(5-2k)v_{17,1}(C_3)\\
   &=& 18k+6+(5-2k)9\\
   &=& 51.
\end{eqnarray*}
Since $v_{17,1}(D_3)=51$, this proves $v_{17,1}(D)=51$.
\end{proof}

\begin{proof}[Proof of Theorem~\ref{teorema impossible}]
Note that $\varphi$ preserves $v_{-13,-1}$ and $v_{17,1}$, since $v_{0,-1}(D_2)\le -13$ and $v_{0,1}(D_2)\le 17$. In particular
$$
51\ge v_{17,1}(x^{3-k}d_{3-k})=51-17k +v_{17,1}(d_{3-k})
$$
 and so $v_{0,1}(d_{3-k})=v_{17,1}(d_{3-k})\le 17k$, hence
\begin{equation}\label{cota para el grado de los d}
  \deg(d_1)\le 34\quad\text{and}\quad  \deg(d_0)\le 51.
\end{equation}
This also implies that
$$
-39\ge v_{-13,-1}(x^{3-k} d_{3-k})=-39+13k+v_{-13,-1}( d_{3-k}),
$$
and so $v_{0,-1}( d_{3-k})=v_{-13,-1}( d_{3-k})\le -13k$, i.e.,
\begin{itemize}
  \item $v_{0,-1}( d_1)\le 26$ and so $y^{26}$ divides $d_1$,
  \item $v_{0,-1}( d_0)\le 39$ and so $y^{39}$ divides $d_0$,
  \item $v_{0,-1}( d_{-1})\le 52$ and so $y^{52}$ divides $d_{-1}$.
\end{itemize}
Define the polynomials $\tilde d_1$, $\tilde d_0$ and $\tilde d_{-1}$ by
$$
d_1=y^{26}\tilde d_1,\quad d_0=y^{39}\tilde d_0\quad\text{and}\quad d_{-1}=y^{52}\tilde d_{-1}.
$$
Then, using that
$$
C_3^2 F_{-4}=f=y(y+1)g(y)\quad\text{and}\quad C_3=y^8(y+1),
$$
we obtain
\begin{align*}
  C_3^{23} d_1 d_{-1}^6 F_{-4} & =(y^8(y+1))^{21} y^{26}\tilde d_1 y^{52\cdot 6}(\tilde d_{-1})^6 y(y+1)g(y)=y^{507}(y+1)^{22}\tilde d_1 (\tilde d_{-1})^6 g(y) \\
  C_3^{69} F_{-4}^3 & = (y^8(y+1))^{63}(y(y+1)g(y))^3=y^{507}(y+1)^{66}g(y)^3\\
  d_0 d_{-1}^9&= y^{507} \tilde d_0 (\tilde d_{-1})^9.
\end{align*}
Hence, from equality~\eqref{ecuacion principal} we obtain

\begin{equation}\label{nueva ecuacion principal}
  0=18 (y+1)^{22}\tilde d_1 (\tilde d_{-1})^6 g(y)+8 (y+1)^{66}g(y)^3+ 27 \tilde d_0 (\tilde d_{-1})^9.
\end{equation}
This implies that $(\tilde d_{-1})^6$ divides $(y+1)^{66}g(y)^3$, and since $g(y)$ is separable, necessarily there exists some $k$ such that $\tilde d_{-1}=(y+1)^k$.

Now we will arrive at the desired contradiction, proving that $k\ge 8$ is impossible and that $k\le 7$ is also impossible:

Assume by contradiction that
$k\ge 8$, then $(y+1)^{70}$ divides $18 (y+1)^{22}\tilde d_1 (\tilde d_{-1})^6 g(y)$ and
$(y+1)^{72}$ divides $27 \tilde d_0 (\tilde d_{-1})^9$. But the multiplicity of $(y+1)$ in  $8 (y+1)^{66}g(y)^3$
is exactly $66$, since $(y+1)$ doesn't divide $g(y)$, so~\eqref{nueva ecuacion principal} yields the contradiction.

On the other hand, if $k\le 7$, then $\deg(\tilde d_{-1})\le 7$ and so, since by~\eqref{cota para el grado de los d} we know that $\deg(\tilde d_1)\le 8$ and
$\deg(\tilde d_0)\le 12$, we obtain
$$
\deg(18 (y+1)^{22}\tilde d_1 (\tilde d_{-1})^6 g(y))\le 22+8+42+4=76
$$
and
$$
\deg(27 \tilde d_0 (\tilde d_{-1})^9)\le 12+9\cdot 7=75,
$$
which is impossible by~\eqref{nueva ecuacion principal}, since $\deg((y+1)^{66}g(y)^3)=78$. This contradiction concludes the proof of Theorem~\ref{teorema impossible}.
\end{proof}

\begin{corollary}
There exist no $P,Q \in K[x,y]$ with $[P,Q] = x$ and
	\begin{align*}
	N(P) = \{(0,0),(1,1),(6,16), (6,18), (0,18)\} \\
	N(Q) = \{(0,0),(1,0),(9,24), (9,27), (0,27)\}
	\end{align*}
\end{corollary}

\begin{proof}
We claim that $\ell_{0,1}(P)=\lambda_p y^18 (x-\lambda)^6$ for some $\lambda_p,\lambda\in K^{\times}$.
If the claim is true, take $\phi\in \Aut(K[x,y])$ with $\phi(y)=y$ and $\phi(x)=x+\lambda$. Then
\begin{equation}\label{sucesores}
\Succ_{\phi(P)}(1,0)\ge (-1,1)\quad\text{and}\quad \Succ_{\phi(Q)}(1,0)\ge (-1,1).
\end{equation}

Since $[\phi(P),\phi(Q)]=x+\lambda$, the polynomials $\phi(P),\phi(Q)$ satisfy the conditions of Theorem~\ref{teorema impossible},
a contradiction which concludes the proof.

In order to prove the claim and~\eqref{sucesores}, consider the map $\psi:K[x,y]\to L^{(2)}$ given by
$\psi(x)=x^{1/2}$ and $\psi(y)=y$. Then $(\psi(P),\psi(Q))$ is an $m,n$-pair for $(m,n)=(2,3)$ (see~\cite{GGV1}*{Definition 4.3}), and
$\en_{1,0}(F)=\frac 2/3 \frac {1}{m}\en_{1,0}(P)$ for $F$ as in~\cite{GGV1}*{Theorem 2.6}. Hence $q=3$ in~\cite{GGV1}*{Corollary 7.2}
and so $\ell_{0,1}(\psi(P))$ is a sixth power, hence so is $\ell_{0,1}(P)$. By the same argument, for $(\rho,\sigma)=\Succ_{\psi(\phi(P))}(0,1)$ we also have
that $\ell_{\rho,\sigma}(\psi(\phi(P)))$ is a sixth power, and since $\st_{\rho,\sigma}(\psi(\phi(P)))=(6,18)$, we know that $(\rho,\sigma)\in\{(-1,1),(-2,1),(-3,1)\}$,
which proves~\eqref{sucesores}.
\end{proof}

\section{Systems of polynomial equations for $(7,21)$}
\label{seccion 5}
In this section we will prove the following theorem.
\begin{theorem}\label{teorema impossible para 7}
There exist no pair of polynomials $P,Q\in K[x,y]$ such that $[P,Q] = x$ and
\begin{align*}
	N(P) = \{(0,0),(4,0),(2,6), (0,14)\} \\
	N(Q) = \{(0,0),(6,0),(3,9), (0,21)\}
\end{align*}	
\end{theorem}

So we assume that such a pair exists, and we will  arrive at a contradiction.
By a linear change of variables we may assume that $\ell_{1,0}(P)=x^6 y^2$ and $\ell_{1,0}(Q)=x^9 y^3$

\begin{proposition}\label{calculo de C para 7}
 There exist $C,F\in K[y,y^{-1}]((x^{-1}))$
with $v_{1,0}(F)=-4$ and $\alpha_2,\alpha_1,\alpha_0,\alpha_{-1}\in K$ such that
  \begin{enumerate}
    \item $P=C^2$ and $Q=C^3 +\alpha_2 C^2+\alpha_1 C+\alpha_0+\alpha_{-1}C^{-1}+F$,
    \item $\ell_{1,0}(C)=x^3 y$,
    \item $v_{2,1}(C)=v_{2,1}(3,1)=7$, if we consider $K[y,y^{-1}]\subset K((y^{-1}))$,
  \item $v_{1,-1}(C)=v_{1,-1}(3,1)=2$, if we consider $K[y,y^{-1}]\subset K((y))$.
  \end{enumerate}
\end{proposition}

\begin{proof}
We will construct inductively
$$
C=x^3 y +x^2 C_2+x C_1 +C_0+ x^{-1}C_{-1}+\dots,
$$
with $C_k\in K[y,y^{-1}]$ such that $C^2=P$ and such that items~(3) and~(4) are satisfied. It suffices to find $C_{3-k}$ for $k\ge 0$ such that
\begin{enumerate}
  \item[$a_k$)] $\left(\left(\sum_{i=0}^{k} C_{3-i}x^{3-i}\right)^2\right)_{6-j}=P_{6-j}$ for $j=0,\dots k$,
  \item[$b_k$)] $v_{2,1}(x^{3-j} C_{3-j})\le 7$ for $j=0,\dots, k$, if we consider $K[y,y^{-1}]\subset K((y^{-1}))$, and
  \item[$c_k$)] $v_{1,-1}(x^{3-j} C_{3-j})\le 2$ for $j=0,\dots, k$, if we consider $K[y,y^{-1}]\subset K((y))$.
\end{enumerate}

We already have
\begin{enumerate}
  \item[$a_0$)] $(x^3 y)^2=x^6 P_6$,
  \item[$b_0$)] $v_{2,1}(x^3 y)=v_{2,1}(3,1)=6$ and
  \item[$c_0$)] $v_{1,-1}(x^3 y)=v_{1,-1}(3,1)=2$.
\end{enumerate}
 Assume that we have constructed $C_3,\dots,C_{3-k+1}\in K[y,y^{-1}]$ such that
\begin{enumerate}
  \item[$a_{k-1}$)] $\left(\left(\sum_{i=0}^{k-1} C_{3-i}x^{3-i}\right)^2\right)_{6-j}=P_{6-j}$ for $j=0,\dots k-1$,
  \item[$b_{k-1}$)] $v_{2,1}(x^{3-j} C_{3-j})\le 7$  for $j=0,\dots, k-1$, if we consider $K[y,y^{-1}]\subset K((y^{-1}))$, and
  \item[$c_{k-1}$)] $v_{1,-1}(x^{3-j} C_{3-j})\le 2$  for $j=0,\dots, k-1$, if we consider $K[y,y^{-1}]\subset K((y))$.
\end{enumerate}
Then we want to find $C_{3-k}$ such that
$$
\left(\left(\sum_{i=0}^{k} C_{3-i}x^{3-i}\right)^2\right)_{6-j}=P_{6-j}\quad\text{for}\quad j=0,\dots, k.
$$
But the term $2C_{3-k}x^{3-k} y x^3=2C_{3-k} y x^{6-k}$ contains the highest power of $x$ whose coefficient
is touched by $C_{3-k}$. Hence it suffices to guarantee
$$
P_{6-k}=\left(\left(\sum_{i=0}^{k} C_{3-i}x^{3-i}\right)^2\right)_{6-k}=\sum_{j=0}^{k}C_{3-j}C_{3-k+j}.
$$
So we want
$$
2y C_{3-k}=P_{6-k}-\sum_{j=1}^{k-1}C_{3-j}C_{3-k+j}
$$
and it suffices to set
\begin{equation}\label{ecuacion de C 3-k para 7}
C_{3-k}:=\frac{1}{2 y}\left( P_{6-k}-\sum_{j=1}^{k-1}C_{3-j}C_{3-k+j}\right)\in K[y,y^{-1}],
\end{equation}
and we obtain that $a_k$) is satisfied.

If we consider $K[y,y^{-1}]\subset K((y^{-1}))$, we have $v_{2,1}(y^{-1})=-1$,
and by induction hypothesis we also have $v_{2,1}(C_{3-j}x^{3-j})\le 7$, hence $v_{2,1}(C_{3-j})+6-2j\le 7$, and so $v_{2,1}(C_{3-j})\le 1+2j$ for $j=1,\dots ,k-1$.
This implies
$$
v_{2,1}(C_{3-j}C_{3-k+j})\le  v_{2,1}(C_{3-j}) v_{2,1}(C_{3-k+j})\le 1+2j+1+2(k-j)=2+2k.
$$
We also have $v_{2,1}(P_{6-k}x^{6-k})\le 14$, hence $v_{2,1}(P_{6-k})+12-2k\le 14$, and so $v_{2,1}(P_{6-k})\le 2+2k$.
Now from ~\eqref{ecuacion de C 3-k para 7} it follows that
$$
v_{2,1}(C_{3-k})\le -v_{2,1}(y)+v_{2,1}\left( P_{6-k}+\sum_{j=1}^{k-1}C_{3-j}C_{3-k+j}\right)\le -1+2+2k=1+2k,
$$
which implies $v_{2,1}(C_{3-k}x^{3-k})\le 7$ and so~$b_k$) holds.

Similarly, if we consider $K[y,y^{-1}]\subset K((y))$, we have $v_{1,-1}(y^{-1})=1$,
and by induction hypothesis we also have $v_{1,-1}(C_{3-j}x^{3-j})\le 2$, hence $v_{1,-1}(C_{3-j})+3-j\le 2$, and so $v_{1,-1}(C_{3-j})\le j-1$ for $j=1,\dots ,k-1$.
This implies
$$
v_{1,-1}(C_{3-j}C_{3-k+j})\le  v_{1,-1}(C_{3-j}) v_{1,-1}(C_{3-k+j})\le j-1+ (k-j)-1=k-2
$$
We also have $v_{1,-1}(P_{6-k}x^{6-k})\le 4$, hence $v_{1,-1}(P_{6-k})+6-k\le 4$, and so $v_{1,-1}(P_{6-k})\le k-2$.
Now from ~\eqref{ecuacion de C 3-k para 7} it follows that
$$
v_{1,-1}(C_{3-k})\le v_{1,-1}(y^{-1})+v_{1,-1}\left( P_{6-k}+\sum_{j=1}^{k-1}C_{3-j}C_{3-k+j}\right)\le 1+k-2=k-1,
$$
which implies $v_{1,-1}(C_{3-k}x^{3-k})\le 2$ and so~$c_k$) holds.

In order to finish the proof it suffices to find $F\in K[y,y^{-1}]((x^{-1}))$
with $v_{1,0}(F)=-4$ such that
\begin{equation}\label{igualdad para Q}
Q=C^3 +\alpha_2 C^2+\alpha_1 C+\alpha_0+\alpha_{-1}C^{-1}+F.
\end{equation}

Since $\ell_{1,0}(Q)=\ell_{1,0}(C^3)$ we have  $v_{1,0}(Q-C^3)<9$. If $v_{1,0}(Q-C^3)>-4$, then
$$
v_{1,0}(Q-C^3)+v_{1,0}(P)-v_{1,0}(1,1)>-4+6-1=1=v_{1,0}(x)= v_{1,0}([Q-C^3,P])
$$
which by~\cite{GGV1}*{Proposition 1.13} implies $[\ell_{1,0}(P), \ell_{1,0}(Q-C^3)]=0$, and so, by~\cite{GGV1}*{Proposition 2.1} and the fact that $x^3 y$ is not the positive
power of any element of $K[x,y]$, we have
$$
\ell_{1,0}(Q-C^3)=\alpha_k (x^3 y)^k,
$$
for some $k$ with $-2<k<3$. Using the arguments of \cite{GGV3}*{Section 1}, we find $\alpha_{2},\alpha_{1},\alpha_{0},\alpha_{-1}\in K$ such that
$$
[\ell_{1,0}(P), \ell_{1,0}(Q-C^3-\alpha_2 C^2-\alpha_1 C-\alpha_0-\alpha_{-1}C^{-1})]\ne 0,
$$
and so, again by~\cite{GGV1}*{Proposition 1.13} we know that  $F:=Q-C^3-\alpha_2 C^2-\alpha_1 C-\alpha_0-\alpha_{-1}C^{-1}$ satisfies
$$
v_{1,0}(F)+v_{1,0}(P)-v_{1,0}(1,1)=v_{1,0}([F,P]),
$$
which implies $v_{1,0}(F)=-4$ and concludes the proof.
\end{proof}

\begin{remark}
  Note that if we replace $Q$ by $\tilde Q := Q - \alpha_2 P - \alpha_0$ and $P$ by $\tilde P:= P+\frac 23 \alpha_1$, then
$$
\tilde Q = C^3 +  \alpha_1 C +  \lambda C^{-1} + F.
$$
Moreover $\tilde C$ with $\tilde C^2=\tilde P$ is given by a series
$$
\tilde C=C+\frac 13 \alpha_1 C^{-1}+\gamma_3 C^{-3}+\gamma_5 C^{-5}+\dots
$$
and so
$$
\tilde C^3=C^3+\alpha_1 C+\tilde\gamma_1 C^{-1}+\tilde\gamma_3 C^{-3}+\dots.
$$
It follows that
$$
\tilde Q=\tilde C^3+\lambda \tilde C^{-1}+\tilde F,\quad\text{for some $\lambda\in K$}.
$$
Since  $\tilde P$, $\tilde Q$ satisfy the conditions of Theorem~\ref{teorema impossible}, we can and will assume that $\alpha_2,\alpha_1,\alpha_0$ vanish in Proposition~\ref{calculo de C}.
\end{remark}

By definition $F_{-4}\in K[y,y^{-1}]$. From
$$
x=[\ell_{1,0}(P),\ell_{1,0}(F)]=[x^6 y^2,F_{-4}x^{-4}]
$$
it follows immediately that $F_{-4}=\frac 12 y^{-1}$.

The equalities
\begin{equation}\label{ecuaciones basicas para 7}
C^2=P\quad\text{and}\quad Q=C^3+\lambda C^{-1}+F
\end{equation}
 yield a system of polynomial equations for $C_k$, similar to the systems of~\cite{GGV3}, which correspond to
the equalities
\begin{equation}\label{ecuaciones polinomiales para 7}
P_{-k}=0\quad\text{for}\quad k=1,\dots,8\quad\text{and}\quad Q_{-k}=0\quad\text{for}\quad k=1,\dots 5.
\end{equation}

However $C_k$ in general is not a polynomial, so we will transform the system into a system for certain $D_k$'s which are polynomials.

\begin{proposition}
Set $D_k:=C_k y^{2-k}$. Then $D_k\in K[y]$.
\end{proposition}

\begin{proof}
  By definition $D_3=1$ and $D_2=C_2 =\frac 1{2y} P_5$. But $v_{1,-1}(P_5 x^5)\le 4$, hence $v_{1,-1}(P_5)\le -1$, which implies that
  $y\mid P_5$. So $D_2\in K[y]$.
  Now assume that $D_j\in K[y]$ for $j>k\le 1$. Then
  $$
  P_{k+3}=\sum_{j=k}^{3} C_j C_{k+3-j}=2y C_k+\sum_{j=k+1}^{2} C_j C_{k+3-j}.
  $$
  Note that $P_k\in K[y]$ for all $k$. In fact, for $k<0$, $P_k=0\in K[y]$.
  Then
  $$
  D_k=y^{2-k}C_k=\frac 12\left(P_{k+3}-\sum_{j=k+1}^{2} C_j C_{k+3-j} \right)y^{1-k}=\frac 12 P_{k+3}y^{1-k}-\sum_{j=k+1}^{2} D_j D_{k+3-j},
  $$
  since
  $$
  D_j D_{k+3-j}=C_j C_{k+3-j}y^{2-j}C_3^{2-(k+3-j)}=C_j C_{k+3-j}y^{4-3-k}.
  $$
  This shows inductively that $D_k\in K[y]$.
\end{proof}
\begin{proof}[Proof of Theorem~\ref{teorema impossible para 7}:]
If we set $D:=\sum_{k\le 3} x^k D_k$, then
$$
(D^2)_{-k}=\sum_{\overset{i,j\le 3}{i+j=-k}}D_i D_j=\sum_{\overset{i,j\le 3}{i+j=-k}}C_i y^{2-i}C_j y^{2-j}=\sum_{\overset{i,j\le 3}{i+j=-k}}C_i C_j y^{4-(i+j)}=(C^2)_{-k}y^{4+k}
$$
and
$$
(D^3)_{-k}=\sum_{\overset{i,j,l\le 3}{i+j+l=-k}}D_i D_jD_l=\sum_{\overset{i,j,l\le 3}{i+j+l=-k}}C_i C_jC_l y^{6-(i+j+l)}=(C^3)_{-k} y^{6+k}.
$$
Now the equalities~\eqref{ecuaciones polinomiales para 7} imply that the following equalities
\begin{equation}\label{ecuaciones polinomiales nuevas para 7}
(D^2)_{-k}=0\quad\text{for}\quad k=1,\dots,8,\quad (D^3)_{-1}=0,\quad (D^3)_{-2}=0,
\end{equation}
$$
(D^3)_{-3}+\lambda y^8=0\quad\text{and}\quad  (D^3)_{-4}-\lambda D_2 y^8+F_{-4}y^{10}=0
$$
hold. Note that we don't use the equality $Q_{-5}=0$.

So there exists $D,\hat F\in K[y]((x^{-1}))$, such that $\ord_{x^{-1}}(\hat F)\ge 5$, $\hat P:= D^2\in K[x,y]$ and
$$
\hat Q:=D^3+D^{-1}\lambda y^8 +x^{-4}F_{-4}y^{10}+\hat F\in K[x,y].
$$
In fact, $P=C^2\in K[x,y]$ implies that $(C^2)_{-k}=0$ for all $k>0$, hence $(D^2)_{-k}=0$ for all $k>0$, so $\hat P\in K[x,y]$;
on the other hand the equalities~\eqref{ecuaciones polinomiales nuevas para 7} imply that
$(\hat Q)_{-k}=0$ for $k=1,2,3,4$. Now set
$$
(\hat F)_{-k}:=-(D^3+D^{-1}\lambda y^8+F_{-4} y^{10})_{-k}\quad\text{for $k\ge 5$}
$$
which implies that $\hat Q\in K[x,y]$, as desired.

Consider the automorphism $\varphi$ of $K[y]((x^{-1}))$ given by $\varphi(y)=y$, $\varphi(x)=x-D_2$ (and so $\varphi(x^{-1})=x^{-1}+x^{-2}D_2+x^{-3}D_2^2+x^{-4}D_2^3+\dots$).
Then $\varphi(\hat P),\varphi(\hat Q)\in K[x,y]$, and if we set $\tilde D:=\varphi(D)$, then
$$
\tilde D=x^3+x d_1 +d_0+x^{-1}d_{-1}+x^{-2}d_{-2}+\dots
$$
for some $d_i\in K[y]$ (note that $d_2=0$), and the conditions $\varphi(\hat P),\varphi(\hat Q)\in K[x,y]$ imply the equalities
$$
(\tilde D^2)_{-k}=0\quad\text{for}\quad k=1,\dots,8\quad\text{and}\quad (\tilde D^3+\tilde D^{-1}\lambda y^8+x^{-4}F_{-4}y^{10})_{-j}=0\quad\text{for}\quad j=1,\dots 4.
$$
Note that we can take $x^{-4}$ instead of $\varphi(x^{-4})$ and that we can ignore the term $\varphi(\hat F)$, since we only consider $j\le 4$.
So we obtain the following equalities:
\begin{align*}
  0=(\tilde D^2)_{-1} & = 2 d_0 d_{-1} + 2 d_1 d_{-2} + 2 d_{-4}\\
0=(\tilde D^2)_{-2} &= d_{-1}^2 + 2 d_0 d_{-2} + 2 d_1 d_{-3} + 2 d_{-5}\\
0=(\tilde D^2)_{-3} &= 2 d_{-1} d_{-2} + 2 d_0 d_{-3} + 2 d_1 d_{-4} + 2 d_{-6}\\
0=(\tilde D^2)_{-4} &= d_{-2}^2 + 2 d_{-1} d_{-3} + 2 d_0 d_{-4} + 2 d_1 d_{-5} + 2 d_{-7}\\
0=(\tilde D^2)_{-5} &= 2 d_{-2} d_{-3} + 2 d_{-1} d_{-4} + 2 d_0 d_{-5} + 2 d_1 d_{-6} + 2 d_{-8}\\
0=(\tilde D^2)_{-6} &= d_{-3}^2 + 2 d_{-2} d_{-4} + 2 d_{-1} d_{-5} + 2 d_0 d_{-6} + 2 d_1 d_{-7} + 2 d_{-9}\\
0=(\tilde D^2)_{-7} &= 2 d_{-1}0 + 2 d_{-3} d_{-4} + 2 d_{-2} d_{-5} + 2 d_{-1} d_{-6} + 2 d_0 d_{-7} +  2 d_1 d_{-8}\\
0=(\tilde D^2)_{-8} &= 2 d_{-11} + d_{-4}^2 + 2 d_{-3} d_{-5} + 2 d_{-2} d_{-6} + 2 d_{-1} d_{-7} + 2 d_0 d_{-8} + 2 d_1 d_{-9}\\
0=(\tilde Q)_{-1}&= 3 d_0^2 d_{-1} + 3 d_1 d_{-1}^2 + 6 d_0 d_1 d_{-2} + 3 d_{-2}^2 + 3 d_1^2 d_{-3} + 6 d_{-1} d_{-3} + 6 d_0 d_{-4} + 6 d_1 d_{-5}\\
&\quad + 3 d_{-7}\\
0=(\tilde Q)_{-2}&= 3 d_0 d_{-1}^2 + 3 d_0^2 d_{-2} + 6 d_1 d_{-1} d_{-2} + 6 d_0 d_1 d_{-3} +  6 d_{-2} d_{-3} + 3 d_1^2 d_{-4} + 6 d_{-1} d_{-4}\\
&\quad + 6 d_0 d_{-5} + 6 d_1 d_{-6} +  3 d_{-8}\\
0=(\tilde Q)_{-3}&= d_{-1}^3 + 6 d_0 d_{-1} d_{-2} + 3 d_1 d_{-2}^2 + 3 d_0^2 d_{-3} + 6 d_1 d_{-1} d_{-3} + 3 d_{-3}^2 + 6 d_0 d_1 d_{-4} + 6 d_{-2} d_{-4}\\
 &\quad + 3 d_1^2 d_{-5} + 6 d_{-1} d_{-5} + 6 d_0 d_{-6} + 6 d_1 d_{-7} + 3 d_{-9}+\lambda y^8\\
0=(\tilde Q)_{-4}&= 3 d_{-10} + 3 d_{-1}^2 d_{-2} + 3 d_0 d_{-2}^2 + 6 d_0 d_{-1} d_{-3} + 6 d_1 d_{-2} d_{-3} +  3 d_0^2 d_{-4} + 6 d_1 d_{-1} d_{-4}\\
&\quad + 6 d_{-3} d_{-4} + 6 d_0 d_1 d_{-5} + 6 d_{-2} d_{-5} +  3 d_1^2 d_{-6} + 6 d_{-1} d_{-6} + 6 d_0 d_{-7} + 6 d_1 d_{-8}\\
&\quad + F_{-4}y^{10}.
\end{align*}
We use that $F_{-4}=\frac 12 y^{-1}$ and we consider this as a system of 9 equations (excluding $(\tilde D^2)_{-8}$, $(\tilde D^2)_{-6}$ and $(\tilde D^3)_{-3}$ )
 and using a CAS (for example Mathematica) we eliminate the variables $d_{-10}, d_{-8}, d_{-7}, d_{-6}, d_{-5}, d_{-4}, d_{-3}, d_{-2}$, obtaining

\begin{equation}\label{ecuacion principal para 7}
  9 y^9 d_1 (d_{-1})^6 +  y^{27} + 27 d_0 (d_{-1})^9=0.
\end{equation}
It follows that $d_{-1}=\alpha y^k$  for some $k\in \mathds{N}_0$ and $\alpha\in K^{\times}$.
A straightforward computation using $\deg_y(d_0)=9$ and $\deg(d_1)=6$, shows that $k=2$.

From the system of equations we also obtain
$$
2 d_{-3} y^9 = 3 d_0 (d_{-1})^4\quad\text{and}\quad  3 (d_{-1})^2 d_{-2} = y^9,
$$
which, together with~\eqref{ecuacion principal para 7}, yield
\begin{align*}
  d_1 & =\frac{-27 d_0 (d_{-1})^9 - y^{27}}{9 (d_{-1})^6 y^9}, \\
  d_{-2} & = \frac{y^9}{3(d_{-1})^2},\\
  d_{-3} & = \frac{3 d_0 (d_{-1})^4}{2 y^9}.
\end{align*}

Finally, inserting these values in
$$
0=-(d_{-1})^3 - 6 d_0 d_{-1} d_{-2} - 3 d_1 (d_{-2})^2 + 3 (d_{-3})^2 + 2 \lambda y^8,
$$
which comes from
$$
0=2(\tilde D^3)_{-3}-3 \left((\tilde D^2)_{-6} + d_1(\tilde D^2)_{-4} + d_0(\tilde D^2)_{-3} + d_{-1}(\tilde D^2)_{-2} + d_{-2} (\tilde D^2)_{-1}\right),
$$
we obtain
$$
\frac{(27 \alpha^9 d_0 - 2 y^9)^2}{108 \alpha^{10}}=\alpha^3 y^8 - 2 \lambda y^{10},
$$
which yields the desired contradiction, since the right hand side is not the square of a polynomial, because $\alpha,\lambda\ne 0$.
\end{proof}

\end{document}